\documentclass[10pt]{article}
\usepackage{amsthm}
\usepackage{thmtools}
\usepackage[colorlinks=true,linkcolor=blue]{hyperref}
\usepackage{enumitem}
\usepackage{mathrsfs}
\usepackage[all]{xy}
\usepackage{caption}
\usepackage{graphics}
\usepackage{graphicx}
\usepackage{tikz}
\usepackage{url}
\usepackage{amsmath}
\usepackage{amsxtra}
\usepackage{amssymb}
\usepackage{lineno}
\usepackage{turnstile}
\usepackage{verbatim}
\usepackage{enumitem}
\usepackage[retainorgcmds]{IEEEtrantools}
\usepackage{accents}
\usepackage[all]{xy}
\usetikzlibrary{matrix}
\usetikzlibrary{calc,decorations.pathmorphing,shapes}

\declaretheoremstyle[bodyfont=\sl]{slanted}

\swapnumbers
\declaretheorem[name=Definition,style=definition,qed=$\dashv$,
numberwithin=section]{dfn}
\declaretheorem[name=Definition,style=definition,numbered=no,qed=$\dashv$]{dfn*}
\declaretheorem[name=Definition,style=definition,numbered=no]{dfnnoqed*}

\declaretheorem[name=Theorem,style=slanted,sibling=dfn]{tm}
\declaretheorem[name=Theorem,style=slanted,numbered=no]{tm*}
\declaretheorem[name=Lemma,style=slanted,sibling=dfn]{lem}
\declaretheorem[name=Corollary,style=slanted,sibling=dfn]{cor}
\declaretheorem[name=Corollary,style=slanted,numbered=no]{cor*}
\declaretheorem[name=Remark,style=definition,sibling=dfn]{rem}

\swapnumbers
\declaretheoremstyle[headfont=\scshape]{claimstyle}
\declaretheorem[name=Claim,style=claimstyle]{clm}

\declaretheorem[name=Claim,style=claimstyle,numbered=no]{clm*}

\declaretheorem[name=Subclaim,style=claimstyle,numbered=no]{sclm*}

\declaretheorem[name=Subsubclaim,style=claimstyle,numbered=no]{ssclm*}

\declaretheoremstyle[headfont=\scshape]{casestyle}

\declaretheorem[name=Case,style=casestyle]{case}

\declaretheorem[name=Subcase,style=casestyle,numberwithin=case]{scase}

\newcommand{\DC}{\mathrm{DC}}

\newcommand{\cd}{\circledast}

\newcommand{\wt}{\widetilde}

\newcommand{\lgcd}{\mathrm{lgcd}}

\newcommand{\compmode}{1}
\newcommand{\compopt}[2]{\ifthenelse{\equal{\compmode}{0}}{#1}{#2}}
\newcommand{\putamin}{\mathrm{putamin}}

\newcommand{\her}{\mathcal{H}}
\newcommand{\sub}{\subseteq}

\newcommand{\inter}{\cap}
\newcommand{\om}{\omega}
\newcommand{\pow}{\mathcal{P}}
\newcommand{\OR}{\mathrm{OR}}

\newcommand{\cut}{\backslash}

\newcommand{\Tt}{\mathcal{T}}

\newcommand{\Uu}{\mathcal{U}}

\newcommand{\rg}{\mathrm{rg}}

\newcommand{\ins}{\trianglelefteq}

\newcommand{\pins}{\triangleleft}

\newcommand{\crit}{\mathrm{cr}}
\newcommand{\rest}{\!\upharpoonright\!}
\newcommand{\com}{\circ}

\newcommand{\lh}{\mathrm{lh}}
\newcommand{\Ult}{\mathrm{Ult}}
\newcommand{\sats}{\models}

\newcommand{\es}{\mathbb{E}}

\newcommand{\Ttvec}{{\vec{\Tt}}}

\newcommand{\pred}{\mathrm{pred}}
\newcommand{\dirlim}{\mathrm{dir lim}}
\newcommand{\id}{\mathrm{id}}

\newcommand{\conc}{\ \widehat{\ }\ }

\DeclareMathOperator{\card}{card}

\newcommand{\Xx}{\mathcal{X}}

\newcommand{\tu}{\textup}

\renewcommand{\cut}{\backslash}

\newcommand{\exit}{\mathrm{ex}}
\newcommand{\eqdef}{=_{\mathrm{def}}}
\newcommand{\tembto}{\hookrightarrow}

\newcommand{\ddd}{\mathrm{ddd}}
\newcommand{\dds}{\mathrm{dds}}

\newcommand{\successor}{\mathrm{succ}}
\newcommand{\passive}{\mathrm{pv}}
\newcommand{\dropset}{\mathscr{D}}
\newcommand{\inflatearrow}{\rightsquigarrow}

\newcommand{\steelurl}{http://math.berkeley.edu/\textasciitilde steel}

\renewcommand{\cut}{\backslash}

\newcommand{\mininflatearrow}{\inflatearrow_{\minim}}
\newcommand{\minim}{\mathrm{min}}
\newcommand{\spc}{\mathrm{spc}}

\newcommand{\D}{\mathrm{D}}

\title{Full normalization for $\kappa^+$-supercompactness}
\author{Farmer Schlutzenberg\footnote{afirstname dot alastname at tuwien dot ac dot at, afirstname dot alastname at gmail dot com}}

\begin{document}

\maketitle

\begin{abstract}
We extend the normalization results of \cite{fullnorm}
to mice at the level of $\kappa^+$-supercompactness:
given a normal iteration strategy $\Sigma$ for such a mouse $M$,
with both $M$ and $\Sigma$ satisfying certain condensation properties,
we extend $\Sigma$ to a strategy $\Sigma^*$ for stacks of normal trees,
such that every iterate via $\Sigma^*$ is in fact a normal iterate via $\Sigma$.
\end{abstract}

\section{Introduction}

The literature on normalization of iteration trees
has so far dealt exclusively with mice with only short extenders on their sequence (see \cite{ACPFMP}, \cite{steel_local_HOD_comp},
 \cite{iter_for_stacks}, \cite{fullnorm},  \cite{jensen_book}, \cite{siskind_steel}), which reach the level of (many) superstrong cardinals. In this paper we extend the methods of (full) normalization to  mice
 at the level of $\kappa^+$-supercompactness; these mice
 include long extenders on their sequence (see \cite{steel_deiser_notes}, \cite{nsp1}, \cite{nsp1fs}, \cite{voellmer},
 \cite{kappa-plus_v3}).
We follow \cite{kappa-plus_v3} in terms of the precise setup of our premice and their fine structure.

Recall that there are two basic kinds of normalization:
(i) what is referred to by Steel in \cite{ACPFMP}
as \emph{full normalization}, or by the author in \cite{fullnorm} simply as \emph{normalization},
and (ii) what is referred to by Steel in \cite{ACPFMP} as \emph{embedding normalization},
and by the author in \cite{iter_for_stacks}
as \emph{normal realization}. In this paper we will only discuss kind (i).

Also, there are various starting points and conclusions one can consider. In this paper, as in \cite{fullnorm} and \cite{iter_for_stacks}, we will work from the existence of an iteration strategy for normal trees, which satisfies a certain condensation property (which we formulate here), and use it to define a strategy for optimal stacks of normal trees (where \emph{optimal} means there is no unnecessary dropping in model or degree).
As in \cite{fullnorm},
all iterates via this induced stacks strategy will in fact be iterates via the original normal strategy, which coincides with the first normal round of the stacks strategy. We will also show that the relevant condensation property (for the normal strategy) follows from the Dodd-Jensen (alternatively, weak Dodd-Jensen) property. As mentioned, there are other kinds of starting hypotheses one might have; for example, one might start with a strategy for (optimal) stacks, provided in some other manner, and want to show that all iterates
via this stacks strategy are in fact iterates via the normal strategy which is given by its first round. This seems to be more the situation in Steel's \cite{ACPFMP}. But the author expects that the methods developed here should be very relevant to and adapt to those needed in other such contexts. This should in turn provide a key to extending the results of \cite{ACPFMP} to models at the level of $\kappa^+$-supercompactness.

Normalization at the $\kappa^+$-supercompact level is, overall, very much like that at the short extender level (at least, regarding the version of normalization we consider here, as  described above). We will establish the following results,
which are direct analogues of results in \cite{fullnorm}.
(In the following statement,
\emph{$m$-standard} premouse
is defined in
 \ref{dfn:m-standard};
 it includes $m$-soundness,
 but adds some more further structural requirements.
 The prefix \emph{optimal-}
(for the strategy $\Sigma^*$)
is as defined in \cite{kappa-plus_v3}
and is like in \cite[\S1.1.5]{iter_for_stacks};
it prevents player I from making artificial drops.
\emph{Minimal inflation condensation} is just the direct generalization of that notion from \cite[***Definition 3.28]{fullnorm}; see Definition \ref{dfn:minimal_inflation_condensation}. And
\emph{drops in Dodd-degree}
and \emph{drops in any way}
are defined in \cite[***Definition 2.38]{kappa-plus_v3}. And the premice
discussed might be pure $L[\es]$-premice, or some other kind such as strategy premice,
with extender sequence as in \cite{kappa-plus_v3}.)

\begin{tm}\label{tm:full_norm_stacks_strategy}
Let $\Omega>\om$ be a regular cardinal. Let $m\in\om\cup\{0^-\}$.
Let $M$ be an $m$-standard
 premouse \tu{(}at the level of $\kappa^+$-supercompactness,
 as in \cite{kappa-plus_v3}\tu{)}.
 Let $\Sigma$
 be an $(m,\Omega+1)$-strategy for $M$ with minimal inflation
 condensation.
 Then there is an optimal-$(m,\Omega,\Omega+1)^*$-strategy
 $\Sigma^*$ for $M$ satisfying the conclusions of \cite[***Theorem 1.1]{fullnorm}.
 That is,
 $\Sigma\sub\Sigma^*$
and for every stack
$\vec{\Tt}=\left<\Tt_\alpha\right>_{\alpha<\lambda}$ via $\Sigma^*$ with a last
model
and with $\lambda<\Omega$,
there is an $m$-maximal successor length tree $\Xx$
on $M$ such that:
\begin{enumerate}[label=--]
\item $\Xx\rest\Omega+1$ is via $\Sigma$,
\item  if $\lh(\Tt_\alpha)<\Omega$ for all $\alpha<\lambda$
then $\lh(\Xx)<\Omega$,
\item
 $M^{\vec{\Tt}}_\infty=M^\Xx_\infty$
and $\deg^{\vec{\Tt}}_\infty=\deg^\Xx_\infty$,
\item  $b^{\vec{\Tt}}$ drops in model \tu{(}degree,  Dodd-degree\tu{)} iff $b^\Xx$ does,
\item if $b^{\vec{\Tt}}$ does not drop in any way
then $i^{\vec{\Tt}}_{0\infty}=i^\Xx_{0\infty}$.
\end{enumerate}
Moreover,
 $\Sigma^*$ is $\Delta_1(\{\Sigma\})$,
uniformly in $\Sigma$,
and if $\card(M)<\Omega$ then $\Sigma^*\rest\her_{\Omega}$
is $\Delta_1^{\her_{\Omega}}(\Sigma\rest\her_\Omega)$,
uniformly in $\Sigma$.
\end{tm}

Just as in \cite[Corollary 1.4***]{fullnorm}, we have the following corollary, using Theorem \ref{tm:wDJ_implies_cond}
and the results of \cite[***\S2.7]{kappa-plus_v3} on weak Dodd-Jensen:

\begin{cor}
Assume $\DC$. Let $\Omega>\om$ be regular. Let $m\in\om\cup\{0^-\}$.
Let $M$ be a countable, $m$-sound, $(m,\Omega,\Omega+1)^*$-iterable pure $L[\es]$-premouse.\footnote{For pure $L[\es]$-premice, $m$-standardness is a consequence of $(m,\om_1,\om_1+1)^*$-iterability.}
Then there is an $(m,\Omega,\Omega+1)^*$-strategy $\Sigma^*$ for $M$,
with first round $\Sigma$, such that $\Sigma,\Sigma^*$
are related as in Theorem \ref{tm:full_norm_stacks_strategy}.
\end{cor}

And as in \cite[Corollary 7.3***]{fullnorm}, we can  deduce:
\begin{cor}\label{cor:project_to_om_strategy}
Let $\Omega>\om$ be regular.
Let $m<\om$ and let $M$
be an $(m+1)$-sound pure $L[\es]$-premouse with $\rho_{m+1}^M=\om$,
and $\Sigma$ be a \tu{(}the unique\tu{)} $(m,\Omega+1)$-strategy
for $M$.
Then there is an optimal-$(m,\Omega,\Omega+1)^*$-strategy
$\Sigma^*$ for $M$
such that every $\Sigma^*$-iterate of $M$
of size ${<\Omega}$ is a $\Sigma$-iterate of $M$.
\end{cor}

As in \cite[after Corollary 7.3]{fullnorm},
the minimal version of
 \cite[Theorem 9.6]{iter_for_stacks}
also goes through in our present context.
Analogous to \cite[Theorem 7.2***]{fullnorm}, we also have the following variant:

\begin{tm}\label{tm:stacks_iterability_2}
 Let $\Omega>\om$ be regular. Let $m\in\om\cup\{0^-\}$, $M$ be $m$-standard
 and $\Sigma$
 be an $(m,\Omega)$-strategy for $M$ with minimal inflation condensation.
 Then there is an
 optimal-$(m,{<\om},\Omega)$-strategy $\Sigma^*$ for $M$
 with $\Sigma\sub\Sigma^*$, such that for each stack
$\Ttvec=\left<\Tt_i\right>_{i<n}$ via $\Sigma^*$
with $n<\om$ and  $\lh(\Tt_i)$ a successor ${<\Omega}$
for each $i<n$,
 there is a tree $\Xx$ via $\Sigma$ \tu{(}so $\Xx$ is $m$-maximal\tu{)}
such that $M^\Ttvec_\infty=M^\Xx_\infty$, $\deg^\Ttvec_\infty=\deg^\Xx_\infty$,
 $b^{\Ttvec}\inter\dropset_{\deg}^{\Ttvec}=\emptyset$
 iff $b^\Xx\inter\dropset_{\deg}^\Xx=\emptyset$,
and $i^{\Ttvec}=i^\Xx$ in case $b^{\Ttvec}\inter\dropset_{\deg}^{\Ttvec}=\emptyset$.
\end{tm}

There is one key difference
between normalization at the level of $\kappa^+$-super-compactness and that for short extenders, in terms of the results of above and their analogues in \cite{fullnorm}.
It arises in the normalization of a stack $(\Tt,\Uu)$ in which $\Tt$ and $\Uu$ each use just one extender, $E$ and $F$ respectively, and where $E$ is short and $F$ is long, with $\kappa=\crit(E)=\crit(F)$,
$E\in\es_+^M$,  $M\sats$ ZFC, $E$ is $M$-total, $F\in\es_+^{U}$ where $U=\Ult_0(M,E)$,
and  $F$ is $U$-total.
 Following the process of normalization for short extenders, we would then consider $\Ult_0(M|\lh(E),F)$, and would like to show that
\[P=\Ult_0(M|\lh(E),F)\pins\Ult_0(M,F)\]
(we can't have equality of the two ultrapowers here,
because $M|\lh(E)\pins M$).
If one used the naive definition for $P=\Ult_0(M|\lh(E),F)$; that is, defining $P$ to have active extender $F'=\bigcup_{\xi<\lh(E)}j(E\cap M|\xi)$,
where $j:M||\lh(E)\to\Ult_0(M||\lh(E),F)$ is the ultrapower map (this is what one always does in the short extender context), then $P$ would  not even be a premouse,
and hence we would not have $P\pins \Ult_0(M,F)$.
(Instead, $P$
would be a proper protomouse,
as $j$ is discontinuous at $\kappa^{+M}=\kappa^{+M|\lh(E)}$, but $\OR^P=\sup j``\lh(E)$,
and so $F'$ only measures sets in $P||\sup j``\kappa^{+M}$, which is not all of $\pow(j(\kappa))\cap P$.)
But following \cite[***Definition 2.19]{kappa-plus_v3},
we do
not define the active extender $F^P$ of $P$
to be $F'$. Instead, we form $F^P$ so as to ``avoid the protomouse'', setting $F^P=F'\com G$,
where $F'$ is as above and $G$ is the short part of $F$ (which was the long extender above).
Note here that we have $G\in\es^{\Ult_0(M||\lh(E),F)}$
by the ISC for $M$. With $P$ defined in this manner,
we will in fact be able to show using condensation (adapting the same kind of argument as that used in the short extender context) that $P\pins\Ult_0(M,F)$, as desired.

Moreover, the usual calculations show that $F^P=F'\com G$ is equivalent to the extender of the stack $(\Tt,\Uu)$. Thus, the normalization of $(\Tt,\Uu)$
is just the tree $\Xx$ which,
if $\lh(E)<\lh(F)$, 
uses $E$, then $F$, then $F^P$,
whereas if $\lh(F)<\lh(E)$,
then $\Xx$ uses just $F$ then $F^P$.

This kind of issue is the main difference
between normalization at the present level and that at the short extender level (where no proper protomice arise in the process). (We must also avoid  proper protomice arising elsewhere as degree $0$ ultrapowers of models in dropdown sequences, but this is a similar process.) The change also means that the correspondence of tree structures between $\Tt$ and $\Xx$ is slightly modified (in comparison with the short extender context), but this is straightforward. Given this preview, what remains is to integrate the sketched modifications into the standard machinery.

The modifications indicated above are primarily relevant in the ``innermost'' details of  normalization -- those relating to dropdown sequences and preservation thereof under ultrapower maps (which we discuss in \S\ref{sec:dropdown_pres},
adapting \cite[\S2.2***]{fullnorm}) and the basic extender commutativity calculations behind normalizing a pair of branch extenders arising from a  stack of two normal trees into a single branch extender of the normalized tree (in \S\ref{sec:ult_comm},
adapting \cite[\S2.3***]{fullnorm}). Some variants of the usual Shift Lemma, which were described in \cite[***\S2.7]{kappa-plus_v3}, are also needed
at some points. Given these modifications, the remaining calculations are very minor modifications of those in \cite{fullnorm}; those modifications just involve (i) adjusting the  rules for determining tree order  in normal trees, in order to match them with those in \cite[***Definition 2.38]{kappa-plus_v3},
and related modifications
(and this already comes up in
the normalization of a pair of branch extenders, just mentioned),
(ii) incorporating  $0^-$ as a possible degree of nodes in an iteration tree, etc, as in \cite{kappa-plus_v3}, in particular \cite[***Definition 2.38]{kappa-plus_v3}
(this also affects the definition of dropdown sequences, but only in an ``organizational'' manner), and thus,
(iii) replacing ``drops in model or degree'' with ``drops of any kind'' (meaning ``drops in model, degree or Dodd-degree'') throughout. Because these modifications are very minor,
there is no point in replicating everything here, so we instead just enumerate the changes that should be made to \cite{fullnorm}.

However, a further point should be mentioned.
In \cite{nsp1}, the notion of a \emph{$0$-maximal}
iteration tree  on $M$ is defined in such a manner that if $M$ is
active short and there is an $M$-total long $E\in\es^M$ with $\crit(E)=\crit(F^M)$
and $[0,\alpha+1]^\Tt$ does not drop
and $\beta=\pred^\Tt(\alpha+1)$, then $M^{\Tt}_{\alpha+1}=\Ult_1(M^\Tt_\beta,E^\Tt_\alpha)$
(instead of using $\Ult_0$, as is traditionally done in the short extender hierarchy). It is of course
important here that we know that $\kappa^{+M}<\rho_1^M$, where $\kappa=\crit(E)$,
which is one of the requirements of \emph{projectum free spaces}. However, it is not clear to the author
whether full normalization works for strategies
that work like this. Even in the case that $M$
arises from a dropping iteration tree,
there seems to be a problem. 
That is, let $\Uu$ be a $k$-maximal
tree on some $k$-sound premouse $P$,
where $k>0$. Suppose that $[0,\infty]^\Uu$
drops, $\deg^\Uu_\infty=0$,
and $M=M^\Uu_\infty$ is as above.
Now say we want to form a $0$-maximal
tree $\Tt$ on $M$, defined as above.
It is unclear to the author whether
we should expect full normalization for the stack $(\Uu,\Tt)$. The problem is that a key feature
of the full normalization process is that the extenders used in $\Uu$ are minimally inflated
by the extenders of $\Tt$; that is, if $E^\Uu_\alpha$
is active on $\exit^\Uu_\alpha$, then we might inflate $E^\Uu_\alpha$ to the active extender of $\Ult_0(\exit^\Uu_\alpha,E^\Tt_0)$. But if instead want
to iterate $M$ at degree $1$ (as far as is possible)
then the extenders $E^\Uu_\alpha$ used on the branch to $M$ at degree $0$ will be ``inflated'' with more functions than just those in $\exit^\Uu_\alpha$.
This seems to lead to at least a significant complication of the usual calculations, and
it is not clear to the author whether there is an analogue.\footnote{For example if there is just
one such extender $E^\Uu_\alpha$ used, and it is just a measure, then there does appear to be an argument,
but it is different to the standard one.
But if $E^\Uu_\alpha$ has many generators, 
it is not so clear.}

But none of the preceding paragraph impacts us here, because we are following \cite{kappa-plus_v3},
and so
our notion of \emph{$0$-maximal}
is just the naive one (except that it uses
the non-naive notion of degree $0$ ultrapowers we have just described).

\subsection*{Acknowledgements}

Gef\"ordert durch die Deutsche Forschungsgemeinschaft (DFG) -- Projektnummer 445387776.
Funded by the Deutsche Forschungsgemeinschaft (DFG, German Research Foundation) -- project number 445387776.

Teilweise gef\"ordert durch die Deutsche Forschungsgemeinschaft (DFG) im Rahmen der Exzellenzstrategie des Bundes und der L\"ander EXC 2044--390685587, Mathematik M\"unster: Dynamik-Geometrie-Struktur.

Editing funded by the Austrian Science Fund (FWF) [10.55776/Y1498].

The author thanks the organizers of the \emph{From $\om$ to $\Omega$} conference, National University of Singapore, 2023,
and  the organizers of the \emph{2nd Irvine Conference on Inner Model Theory}, UC Irvine, 2023,
for the opportunity to present this work.

\subsection{Notation}\label{subsec:notation}

We follow \cite{kappa-plus_v3}
for terminology and notation on premice and fine structure,
in particular in connection with premice at the level of $\kappa^+$-supercompactness, which are our focus in this paper.
Other more general notation and terminology is as in \cite{fullnorm}, \cite[\S1.1]{iter_for_stacks} and \cite[\S1.1]{premouse_inheriting}.

For an active premouse $N$, $\lgcd(N)$ denotes the largest cardinal of $N$. For an
extender $E$, $\nu(E)$ denotes the natural length of $E$.
Let $\Tt$ be an $m$-maximal tree. If $\alpha+1<\lh(\Tt)$,
then $\exit^\Tt_\alpha$ denotes $M^\Tt_\alpha|\lh(E^\Tt_\alpha)$ (\emph{ex} for \emph{exit} (model)).
 
\section{Dropdown preservation}\label{sec:dropdown_pres}

\begin{rem}\label{rem:degree_omega_avoided}
The comments on degree $\om$ versus degree $0$
in \cite[\S2.1]{fullnorm}
continue to hold in the present context,
and we will follow the same convention as mentioned there,
hence considering only degrees in $\om\cup\{0^-\}$.
\end{rem}

Just as in \cite[Definition 2.10]{premouse_inheriting} and \cite{fullnorm}, we abstract out some condensation we need to assume
holds of the base premouse $M$ we will be iterating.
Such properties hold of sufficiently iterable premice;
see the condensation results in \cite{kappa-plus_v3} (which are related to those in \cite{voellmer} and
of course elsewhere), especially \cite[Theorems 3.17, 3.18]{kappa-plus_v3}.
The notion \emph{$n$-standard}
below enumerates some basic properties of $M^\Tt_\infty$
if $\Tt$ is a $k$-maximal
tree on a $k$-sound premouse $N$, and there is a drop of some kind along $b^\Tt$.

\begin{dfn}\label{dfn:m-standard}
Let $m<\om$ and let $M$ be an $(m+1)$-sound
premouse such that if $M$ is active short then $M$ is
Dodd-absent-sound. We define \emph{$(m+1)$-relevantly condensing} and \emph{$(m+1)$-sub-condensing} (as applied to $M$) just as in \cite[***Definition 2.1]{fullnorm},
except that we assume that for the premice $P$ considered,
if $P$ is active short then $P$ is Dodd-absent-sound.

Say that $M$ is \emph{short-Dodd-sub-condensing}
iff \emph{if} $M$ is active short
\emph{then} for all $P,\pi$, if
\begin{enumerate}
 \item 
$P$ is active short and Dodd-absent-sound,
\item
$\max(\rho_{\D}^P,\mu^{+P})$ is an $M$-cardinal,
where $\mu=\crit(F^P)$,
\item 
$\pi:P\to M$ is $0$-deriving,\footnote{See \cite{kappa-plus_v3} for the general definition of \emph{$0$-deriving}. In case $\crit(\pi)>\crit(F^P)$,
as is required by condition \ref{item:crit(pi)>=mu^+P,rho_D^P}, this is equivalent to requiring that $\pi$ be $\Sigma_0$-elementary with respect to the language of premice without $\dot{F}_J$.}
\item\label{item:crit(pi)>=mu^+P,rho_D^P} $\crit(\pi)\geq\max(\mu^{+P},\rho_{\D}^P)$, and
\item $\max(\rho_{\D}^P,\mu^{+P})<\rho_1^M$,
\end{enumerate}
 then $P\pins M$.

For $n\in\om\cup\{0^-\}$, a premouse $N$ is  \emph{$n$-standard} iff:
\begin{enumerate}[label=--]
\item $N$ is $n$-sound (recall that $0^-$-soundness is the empty requirement),
\item if $n=0^-$ then  $N$ is  active short,
\item if $N$ is active short
and $n\in\om$ then $N$ is Dodd-absent-sound,
\item if $n>0$ then $N$ is $(m+1)$-relevantly-condensing for every $m<n$, and
\item every $M\pins N$ is short-Dodd-sub-condensing and $(m+1)$-relevantly-condensing
and $(m+1)$-sub-condensing for each $m<\om$.\qedhere
\end{enumerate}
\end{dfn}

\begin{lem}
Let $n\in\om\cup\{0^-\}$.
Let $N$ be an $n$-sound,
$(n,\om_1,\om_1+1)^*$-iterable premouse. Suppose that if $N$ is active short and $n\neq 0^-$ then $N$ is Dodd-absent-sound.
Then $N$ is $n$-standard.
\end{lem}
\begin{proof}
This follows  from
condensation results in \cite{kappa-plus_v3}, in particular \cite[Theorems 3.17, 3.18]{kappa-plus_v3} (\cite[Theorem 3.18]{kappa-plus_v3} for the short-Dodd-sub-condensing aspect).\end{proof}

\begin{rem}\label{rem:condensation_def}
Remarks similar to those in \cite[Remark 2.2]{fullnorm}
hold here, though the analogue of \cite{fsfni_v5}
has not been worked out at the plus one level (nor for $\lambda$-indexing at the short level),
so the known proofs of condensation-like properties mostly rely on $(n,\om_1,\om_1+1)^*$-iterability.\end{rem}

\begin{dfn}
Let $N$ be an active short premouse. Then $\rho_{0^-+1}^N$ denotes $\max(\kappa^{+N},\rho_{\D}^N)$ where $\kappa=\crit(F^N)$.
\end{dfn}

We next define \emph{$(M,m)$-good}
and \emph{$(M,m)$-pre-good} sequences
$\vec{E}$ of extenders.
These hold of the sequence of extenders used along the main branch $b^\Tt$ of an $m$-maximal iteration tree $\Tt$, if there is no dropping of any kind along $b^\Tt$ (and more generally, along the tail of $b^\Tt$ within which there are no drops of any kind, if $m=\deg^\Tt_\infty$).

\begin{dfn}\label{dfn:(M,m)-good} Let $m\in\om\cup\{0^-\}$ and  $M$ be an $m$-standard premouse.
Let $\vec{E}=\left<E_\alpha\right>_{\alpha<\lambda}$ be a sequence of extenders.
We say that $\vec{E}$ is \emph{$(M,m)$-pre-good} iff
there is a sequence $\left<M_\alpha\right>_{\alpha\leq\lambda}$
such that:
\begin{enumerate}[label=--]
 \item $M_0=M$,
 \item for each $\alpha<\lambda$, we have:
 \begin{enumerate}
 \item $M_\alpha$ is an $m$-sound premouse,
 \item $E_\alpha$ is
 a weakly amenable  $M_\alpha$-extender, as defined in \cite[Definition 2.10***]{kappa-plus_v3},\item  letting $\gamma^{M_\alpha}=\lgcd(M_\alpha)$ if $M$ has a largest cardinal, and  $\gamma^{M_\alpha}=\OR^{M_\alpha}$
 otherwise, either:
 \begin{enumerate}
 \item $M$ is active short and $m=0^-$ and $\spc(E_\alpha)<\gamma^{M_\alpha}$, or
 \item $M$ is active short and $m=0$ and $\spc(E_\alpha)<\rho_{0^-+1}^{M_\alpha}=\max(\kappa^{+M_\alpha},\rho_{\D}^{M_\alpha})$, where $\kappa=\crit(F^{M_\alpha})$, or
 \item $M$ is active short and $m>0$
 and $\spc(E_\alpha)<\rho_m^{M_\alpha}$,
 or
 \item $M$ is not active short,
 and $\spc(E_\alpha)<\min(\rho_m^{M_\alpha},\gamma^{M_\alpha})$.
 \end{enumerate}
 \item  $M_{\alpha+1}=\Ult_m(M_\alpha,E_\alpha)$,
 \item if $E_\alpha$ is long then the short part of $E_\alpha$ is in $\es^{M_{\alpha+1}}$,
\end{enumerate}
 and
 \item for each limit $\gamma\leq\lambda$, 
$M_\gamma=\dirlim_{\alpha<\gamma}M_\alpha$, under the (compositions and direct 
limits of) the ultrapower maps.
\end{enumerate}
We write $\Ult_m(M,\vec{E})=M_\lambda$ and $i^{Mm}_{\vec{E}}$ for the
ultrapower map.
We say  $\vec{E}$ is \emph{$(M,m)$-good} iff $\vec{E}$ is $(M,m)$-pre-good
and $M_\lambda$ is wellfounded.
Suppose $\vec{E}$ is $(M,m)$-pre-good. For $\kappa\leq\rho_m^M$,  say
$\vec{E}$ is \emph{${<\kappa}$-space-bounded}
iff
$\spc(E_\alpha)<\sup i^{M,m}_{\vec{E}\rest\alpha}``\kappa$
 for each $\alpha<\lambda$;
 if $\kappa<\rho_m^M$, say
 $\vec{E}$ is \emph{$\kappa$-space-bounded}
 iff it is ${<(\kappa+1)}$-space-bounded.

We say  $\vec{E}$ is \emph{$(M,m)$-pre-pre-good}
iff either $\vec{E}$ is $(M,m)$-pre-good or there is $\gamma<\lh(\vec{E})$
such that $\vec{E}\rest\gamma$ is $(M,m)$-pre-good but not $(M,m)$-good.
\end{dfn}

\begin{lem}\label{lem:ult_pres_standard}
 Let $N$ be $n$-standard and $\vec{E}$ be $(N,n)$-good.
  Then
 $\Ult_n(N,\vec{E})$ is $n$-standard.
 \end{lem}
\begin{proof}
This is a corollary of Remark \ref{rem:condensation_def},
together with the fact that if $N$ is active short and $n=0$, so $N$ is Dodd-absent-sound,
then $\Ult_n(n,\vec{E})$
is Dodd-absent-sound.
This is shown in \cite[***\S2.4]{kappa-plus_v3}. Note that it is possible here that, for example, $\lh(\vec{E})=1$
and $E_0$ is long with $\crit(E_0)=\crit(F^N)$,
in which case $\Ult_0(N,E_0)$
is formed so as to avoid the protomouse. But by $0$-goodness, $\spc(E_0)<\rho_{\D}^N$, so by calculations in \cite{kappa-plus_v3},
$\Ult_0(N,E_0)$ is Dodd-absent-sound. (Also cf.~the proof of Lemma \ref{lem:Ult_m_seg_Ult_n}.)
 \end{proof}

\begin{rem}
In the context of Lemma \ref{lem:ult_pres_standard},
it is possible that $N$ is active short, $n=0$ and $i^{N,0}_{\vec{E}}:N\to\Ult_n(N,\vec{E})$ fails to be a $0$-embedding (and likewise when $n=0^-$).
In fact, $i^{N,0}_{\vec{E}}$ fails to be a $0$-embedding iff there is some $\alpha<\lh(\vec{E})$ such that $\Ult_0(N_\alpha,E_\alpha)$ is formed avoiding the protomouse
(that is, $E_\alpha$ is long and $\crit(E_\alpha)=\crit(F^{N_\alpha})$).
This situation does not arise in the short extender context.
This can also be the case for iteration maps $i^{\Tt}_{\alpha\beta}:M^\Tt_\alpha\to M^\Tt_\beta$ for $0$-maximal trees on such $N$.
However, if $\Tt$ is a $k$-maximal
tree on a $k$-standard premouse $M$,
and $\alpha+1\leq^\Tt\beta$
and $(\alpha+1,\beta]^\Tt$ does not drop in any way,
and $\deg^\Tt_\beta=0$,
but there \emph{is} a drop of some kind in $[0,\beta]^\Tt$,
then $i^{*\Tt}_{\alpha+1,\beta}$ is a $0$-embedding.
(For otherwise $M^\Tt_\beta$ is active short and there is some $\gamma+1\in[\alpha+1,\beta]^\Tt$ such that $\Ult_0(M^{*\Tt}_{\gamma+1},E^\Tt_\gamma)$ is formed avoiding the protomouse.
Since $[0,\gamma+1]^\Tt$ drops in some way, this implies that $\rho_1(M^{*\Tt}_{\gamma+1})\leq\kappa^{+M^{*\Tt}_{\gamma+1}}$,
where $\kappa=\crit(E^\Tt_\gamma)$.
But then $\rho_{\D}(M^{*\Tt}_{\gamma+1})=0$ and $\deg^\Tt_{\gamma+1}=0^-$, a contradiction.)
\end{rem}
Analogous to \cite[Lemma 2.5]{fullnorm}, we have:

\begin{lem}\label{lem:Ult_m_seg_Ult_n}
Let $N$ be $n$-standard and $0\leq m<n<\om$ with $\rho_{m+1}^N=\rho_n^N$.
Let  $\vec{E}$ be an $(N,n)$-good sequence.
Then $\vec{E}$ is $(N,m)$-good.
Let $U_k=\Ult_k(N,\vec{E})$ for $k\in\{m,n\}$.
Then
$U_m\ins U_n$ and $\rho_{m+1}^{U_m}=\rho_{n}^{U_n}$.
\end{lem}
\begin{proof}
For simplicity assume that $\lh(\vec{E})=1$
and $n=m+1$.
The proof  is like that of \cite[Lemma 2.5]{fullnorm}, except for the case that $N$ is active short,
$m=0$, so $n=1$, and letting $\kappa=\crit(F^N)$,
then $\spc(E)=\kappa^{+N}$, so $\kappa^{+N}<\rho_1^N$.
Here $U_0=\Ult_0(N,E)$ is formed so as to avoid the protomouse, whereas $U_1=\Ult_1(N,E)$ is formed in the usual manner. Let
$\sigma:U_0\to U_1$
be the factor map. Note $\sigma$ is not $0$-lifting,
since $\crit(F^{U_0})=\kappa<\crit(F^{U_1})=\lambda(E)<\sup i``\rho_1^{U_0}\leq\crit(\sigma)$
where $i:N\to U_0$ is the ultrapower map.

Now $\kappa^{+N}<\rho_1^N$ and $N$ is $1$-sound and Dodd-absent-sound. So $\rho=\rho_1^N=\rho_{\D}^N$,
 $F^N$ is generated by $\rho\cup\{p_{\D}^N\}$
and $N$ is Dodd-absent-solid.
Recall that by \cite[***Lemma 2.20]{kappa-plus_v3},
$U_0$ is a premouse, and in particular,
satisfies the Jensen ISC.
Let $U'$ be the protomouse associated to $U_0$
(formed by the naive version of $\Ult(N,E)$).
Let $F'=F^{U'}$.

\begin{clm}\label{clm:F^U_0_gend_by}
$F^{U_0}$ is generated by $(\sup i``\rho)\cup\{i(p_{\D}^N)\}$.\end{clm}
\begin{proof}
We first observe that
\begin{equation}\label{eqn:rg(i)_generated}\rg(i)=\big\{i^{U_0,0}_{F^{U_0}}(f)(a,i(p_{\D}^N))\bigm|f\in U_0|\kappa^{+U_0}\text{ and }a\in(i``\rho)^{<\om}\big\}.\end{equation}
For given any $x\in N$, letting $f\in N|\kappa^{+N}$ and $a\in[\rho]^{<\om}$ be such that
\[ x=i^{N,0}_{F^N}(f)(a,p_{\D}^N),\]
then (recalling $F'=F^{U'}$, the active protomouse extender)
\[ i(x)=i^{U_0,0}_{F'}(i(f))(i(a),i(p_{\D}^N)).\]
But $i(f)=i^{U_0,0}_{F^{U_0}}(f)\rest i(\kappa)$,
so $i^{U_0,0}_{F'}(i(f))=i^{U_0,0}_{F^{U_0}}(f)$, 
so
\[ i(x)=i^{U_0,0}_{F^{U_0}}(f)(i(a),i(p_{\D}^N)), \]
showing the ``$\sub$'' direction of line  (\ref{eqn:rg(i)_generated}).
The converse
direction is clear,
since $i^{U_0,0}_{F^{U_0}}(f)=i(i^{N,0}_{F^N}(f))\in\rg(i)$ for each $f$ as there.
But every point in $U_0$ is easily produced from elements of $\rg(i)\cup G$ where $G$ is the set of generators of $E$, and $G\sub i(\kappa^{+M})=i(\kappa)^{+U_0}$, and so every point in $U_0$ is generated by $(i``\rho)\cup\{i(p_{\D}^N)\}\cup i(\kappa)^{+U_0}$, and hence by $(\sup i``\rho)\cup\{i(p_{\D}^N)\}$, as desired.\end{proof}

\begin{clm}\label{clm:U_0_is_Dodd-absent-sound}
$U_0$ is Dodd-absent-sound\footnote{For the other cases, as in the usual proof, one shows directly that $\Ult_m(N,E)$ is $(m+1)$-sound
(with $\rho_{m+1}^{\Ult_m(N,E)}=\sup i``\rho_{m+1}^N$. But in the case under consideration, the author does not know any direct proof that $\Ult_0(N,E)$ is $1$-sound.
 This is discussed in more detail in \cite[***just after Lemma 2.20]{kappa-plus_v3}. (But we will end up showing that $U_0=\Ult_0(N,E)\pins U_1$, so $U_0$ is in fact $1$-sound.)} with
 $\rho_1^{U_0}=\rho_{\D}^{U_0}=\sup i``\rho=\rho_1^{U_1}$ and $p_{\D}^{U_0}=i(p_{\D}^N)$.
\end{clm}
\begin{proof}
Since $\rho$ is an $N$-cardinal
and $\spc(E)<\rho$,
$\sup i``\rho$ is a $U_0$-cardinal.
If $\bar{F}\in N$ is an $N$-total fragment of $F^N$ then $i(\bar{F})$
is compatible with $F'$.
Thus, letting $\bar{E}$ be the short segment of $E$
(so $F^U=F'\com\bar{E}$),
then $i(\bar{F})(\bar{E})$ agrees with $F^{U_0}$ in the obvious manner. This yields the Dodd-absent-solidity witnesses corresponding to the claimed value of $\rho_{\D}^{U_0}$ and $p_{\D}^{U_0}$.
(For example, if $p_{\D}^N\neq\emptyset$,
then letting $\alpha=\max(p_{\D}^n)$,
and letting $\bar{F}=F^N\rest\alpha\in N$, 
then $i(\bar{F})(\bar{E})\rest i(\alpha)=F^{U_0}\rest\alpha\in U_0$.) But this suffices, by Claim \ref{clm:F^U_0_gend_by}. (The fact that $\rho_1^{U_0}=\rho_{\D}^{U_0}$ is easy, and shown in \cite[***Lemma 3.8]{kappa-plus_v3}.)

Finally, we have $\sup i``\rho=\rho_1^{U_1}$ since $i\rest\rho=j\rest\rho$ where $j:N\to\Ult_1(N,E)$
is the degree $1$ ultrapower map, and $\rho=\rho_1^N$.
\end{proof}

Now let $Q\pins U_1$ be the segment with $F^Q=\bar{E}$ (as above). So $\lambda<\OR^Q<\lambda^{+U_1}$
and $\rho_\om^Q=\lambda$ 
where $\lambda=\lambda(E)=i^{N,1}_{E}(\kappa)=\crit(F^{U_1})=\crit(F')$.
Also $Q\pins U_0$.
Let $Q^*=i^{U_1,0}_{F^{U_1}}(Q)$.
So $\lambda(F^{U_1})<\OR^{Q^*}<\OR^{U_1}$
and $\rho_\om^{Q^*}=\lambda(F^{U_1})$.
Let
$\sigma':U'\to U_1$
be the natural factor map
(so $\sigma,\sigma'$ have the same graph).
In particular, $\sigma'\com i'=i^{N,1}_E$
where $i':N\to U'$ is the naive ultrapower map
(this has the same graph as does the ultrapower map $i:N\to U_0$),
and $\crit(\sigma')\geq\sup i``\rho$.

Note  $\rg(\sigma')\sub Q^*$,
and in fact $\sigma'``\OR^{U'}$ is cofinal in $\OR^{Q^*}$.

Recalling  $(U')^{\passive}=(U_0)^{\passive}$, let
$\pi:U_0\to Q^*$
be the map whose graph is the same as those of $\sigma$ and $\sigma'$.
Because $F^{U_0}=F'\com F^Q$,
it is easy enough to see that
 $\pi$ is $0$-deriving.
 But $\rho_\om^{Q^*}=\lambda(F^{Q^*})=\lambda(F^{U_1})$, since this is a cardinal in $U_1$. Since $\rho_{\D}^{U_0}=\rho_1^{U_1}$ is a $U_1$-cardinal, and hence a $Q^*$-cardinal,
 and since $U_1$ is $0$-standard, it follows that $U_0\pins Q^*\pins U_1$, which suffices.
\end{proof}

The following is much like \cite[Definition 2.6]{fullnorm}, but  incorporating the degree $0^-$ for active short levels:
\begin{dfn}
Let $m,n\in \om\cup\{0^-\}$. Let $N$ be an $n$-standard premouse and $(M,m)\ins(N,n)$, where  if $m=0^-$ then $M$ is active short.
 The \emph{extended $((N,n),(M,m))$-dropdown}
 is the sequence $\left<(M_i,m_i)\right>_{i\leq k}$, with $k$ as large as 
possible,
 where $(M_0,m_0)=(M,m)$, and $(M_{i+1},m_{i+1})$ is the least 
$(M',m')\ins(N,n)$ such that either
 \begin{enumerate}[label=--]
  \item  $(M',m')=(N,n)$, or
  \item $(M_i,m_i)\pins(M',m')$ and $\rho_{m'+1}^{M'}<\rho_{m_i+1}^{M_i}$ and
 if $m'=0^-$ then $M'$ is active short.
  \end{enumerate}

 The \emph{reverse extended $((N,n),(M,m))$-dropdown} is 
$\left<(M_{k-i},m_{k-i})\right>_{i\leq k}$.

Abbreviate \emph{reverse extended} with \emph{revex}
and  \emph{dropdown} with 
\emph{dd}.
\end{dfn}
Analogous to \cite[Lemma 2.7]{fullnorm}, we have:

\begin{lem}\label{lem:ult_dropdown}
Let $n\in\om\cup\{0^-\}$
and let $N$ be $n$-standard,
where $N$ is active short if $n=0^-$.
Let $(M,m)\ins(N,n)$.
Let $\left<(M_i,m_i)\right>_{i\leq k}$ be the extended $((N,n),(M,m))$-dropdown.
Let $\vec{E}=\left<E_\alpha\right>_{\alpha<\lambda}$ be a
sequence
such that:
\begin{enumerate}[label=--]
\item $\vec{E}$ is $(M_i,m_i)$-good for each $i\leq k$,
\item $\vec{E}$ is $(M_i,m_i+1)$-good for each $i<k$ with $m_i\neq 0^-$,
\item $\vec{E}$ is $(M_i,0)$-good for each $i<k$ with $m_i= 0^-$,\footnote{Recall that if $m_i=0^-$, \emph{$(M_i,0)$-good} deals with $\rho_{0^-+1}^{M_i}$.}
\end{enumerate}
Let $U_i=\Ult_{m_i}(M_i,\vec{E})$, $M'=U_0$
and $N'=U_k$.
 Then $\left<(U_i,m_i)\right>_{i\leq k}$ is
  the extended $((N',n),(M',m))$-dropdown.
\end{lem}
\begin{proof}
We may assume $k=0$. So suppose $k>0$. Note that it suffices to
prove the following, for each $i<k$:
\begin{enumerate}
 \item\label{item:U_i,m_i_pins_U_i+1,m_i+1} $(U_i,m_i)\pins (U_{i+1},m_{i+1})$.
 \item\label{item:dropdown_of_consec_pair} the extended
$((U_{i+1},m_{i+1}),(U_i,m_i))$-dropdown is 
$\left<(U_i,m_i),(U_{i+1},m_{i+1})\right>$,
 \item\label{item:strict_drop_pres}if $i+1<k$
 then
$\rho_{m_{i+1}+1}^{U_{i+1}}<\rho_{m_i+1}^{U_i}$.

\end{enumerate}

We just give the proof assuming that $\lambda=\lh(\vec{E})=1$;
the general case is then a straightforward induction on $\lh(\vec{E})$.
Let $E=E_0$ and $\delta=\spc(E)$. Note that for each $i<k$, we have
$\delta<\delta^{+N}=\delta^{+M_i}\leq\rho_{m_i+1}^{M_i}$,
and  $\delta^{+N}\leq\rho_n^N$.
Write $R=M_i$, $r=m_i$, $S=M_{i+1}$, $s=m_{i+1}$,
$R'=U_i$, and $S'=U_{i+1}$. So $(R,r)\pins(S,s)$.

To start with we prove clauses
\ref{item:U_i,m_i_pins_U_i+1,m_i+1}--\ref{item:strict_drop_pres}
assuming that $i+1<k$.

\begin{case} $R\pins S$.
 
 Let $\rho=\rho_{r+1}^R$. So $\rho=\rho_\om^R$ is a cardinal of $S$
and $\rho_{s+1}^S<\rho\leq\rho_s^S$
(and if $S$ is active short
and $s=0$ then $\rho_1^S<\rho\leq\rho_{0^-+1}^S$). Therefore the functions
$[\delta]^{<\om}\to\alpha$, for $\alpha<\rho$,
which are used in forming $R'=\Ult_r(R,E)$, are just those in $R|\rho$,
and also just those used in forming
$S'=\Ult_s(S,E)$.
So  letting $i^R:R\to R'$ and $i^S:S\to S'$ be the ultrapower maps and
$\pi:R'\to i^S(R)\pins S'$
the natural factor map, we have
\[ \rho_{r+1}^{R'}=\sup i^R``\rho=\sup i^S``\rho\leq\crit(\pi)\]
and
$\rho_{r+1}^{R'}\leq\sup i^S``\rho_s^S=\rho_s^{S'}$
(and  if $S$ is active short
and $s=0$ then $\rho_{r+1}^{R'}\leq\sup i^S``\rho_{0^-+1}^S=\rho_{0^-+1}^{S'}$),
 as in the proof of Lemma \ref{lem:Ult_m_seg_Ult_n}.
Moreover, $\rho_{r+1}^{R'}$ is a cardinal of $S'$, because $\rho$ is a cardinal 
of $S$,
and if $\rho=\gamma^{+S}$ then $\rho$ is regular in $S$
(and  as $\delta<\rho$,
even if $E$ is long, the ultrapower maps are  continuous at $\rho$).

\begin{scase} $R$ is active short, $r\in\{0,0^-\}$,
and $E$ is long with $\crit(E)=\crit(F^R)$.
Note that since $\delta=\spc(E)<\rho_1^R=\rho_\om^R$ (in fact $\spc(E)<\rho_{s+1}^S<\rho_1^R$) we have $\delta=\kappa^{+R}<\rho_1^R=\rho_{\D}^R$ (so in fact $r=0^-$).
Let $Q^*=i^{S}(F^R)(Q)$ where $Q\pins S'$
has $F^Q=\bar{E}=$ the short part of $E$.
Like in the proof of Lemma \ref{lem:Ult_m_seg_Ult_n},
there is a $0$-deriving  embedding
$\pi':R'\to Q^*$
whose graph is just that of $\pi$.
As there, and since $S'$ is $0$-standard
by  \ref{lem:ult_pres_standard}, we then get $R'\pins Q^*\pins S'$.
\end{scase}

\begin{scase} Otherwise.

Note that  either
\begin{enumerate}[label=--]
 \item $\pi$ satisfies the requirements for
 $(r+1)$-relevant
(if $\pi``\rho_r^R$ is bounded in $\rho_r^{i^S(R)}$),
or
\item $\pi$ satisfies the requirements for $(r+1)$-sub-condensing
(if $\pi``\rho_r^R$ is unbounded in $\rho_r^{i^S(R)}$ but $\rho_{r+1}^R<\rho_{r+1}^{i^S(R)}$),
 or
 \item $R'=i^S(R)$ and $\pi=\id$ (if $\pi``\rho_r^R$ is unbounded in $\rho_r^{i^S(R)}$ and $\rho_{r+1}^R=\rho_{r+1}^{i^S(R)}$).
 \end{enumerate}
But by Lemma \ref{lem:ult_pres_standard}, $S'$ is $0$-standard,
so
 $R'\ins i^S(R)\pins S'$.
\end{scase}

In both subcases,  $\rho_{s+1}^{S}<\rho$ and $\rho_{s+1}^{S'}=\sup i^S``\rho_{s+1}^S$,
so $\rho_{s+1}^{S'}<\rho_{r+1}^{R'}\leq\rho_s^{S'}$.
So parts \ref{item:U_i,m_i_pins_U_i+1,m_i+1}--\ref{item:strict_drop_pres} for 
this case  follow.
\end{case}

\begin{case} $R=S$.

So $r<s$, and note that either:
\begin{enumerate}[label=--]
\item $s>0$ and $\rho_{s+1}^S<\rho_s^S=\rho_{r+1}^S$,
or
\item  $s=0$, $r=0^-$, $S$ is active short, and letting $\kappa=\crit(F^S)$,
we have
\[ \rho_1^S<\rho_{0^-+1}^S=\kappa^{+S}<\rho_{0^-}^S=\rho_0^S=\OR^S.\]
\end{enumerate}
Lemma \ref{lem:Ult_m_seg_Ult_n} gives $R'\ins S'$ and note  that if $s>0$ then
\begin{equation}\label{eqn:R=S_ult_drop} 
\rho_{s+1}^{S'}<\rho_s^{S'}=\rho_{r+1}^{R'}\eqdef\rho',
\end{equation}
whereas if $s=0$ then $R'=S'$
(since $\Ult_0$ and $\Ult_{0^-}$ are identical) and
\begin{equation}\label{eqn:R=S_ult_drop_s=0}
\rho_{1}^{S'}<\rho_{0^-+1}^{S'}=(\kappa')^{+S'}\eqdef\rho'
\end{equation}
where $\kappa'=\crit(F^{S'})$.
Note that $\left<(R',r),(S',s)\right>$
is the extended dropdown of $((S',s),(R',r))$:
If $R'=S'$ this follows from lines (\ref{eqn:R=S_ult_drop})
and (\ref{eqn:R=S_ult_drop_s=0}) above;
if $R'\pins S'$ it is by line (\ref{eqn:R=S_ult_drop}) and because
$\rho'=\rho_{r+1}^{R'}=\rho_\om^{R'}=\rho_s^{S'}$ is a cardinal of $S'$.
\end{case}

This completes the proof of parts 
\ref{item:U_i,m_i_pins_U_i+1,m_i+1}--\ref{item:strict_drop_pres}
assuming that $i+1<k$.

Now consider the case that $i=k-1$.

\begin{clm}Suppose $M_{k-1}\pins N$. Then $\rho_{m_{k-1}+1}^{M_{k-1}}=\rho_\om^{M_{k-1}}$ is an $N$-cardinal, $U_{k-1}\pins N'$ and $\rho_{m_{k-1}+1}^{U_{k-1}}=\rho_\om^{U_{k-1}}$ is an $N'$-cardinal. Clauses \ref{item:U_i,m_i_pins_U_i+1,m_i+1} and \ref{item:dropdown_of_consec_pair} hold for $i=k-1$, and  moreover:
\begin{enumerate}[label=(\alph*)]
\item If $n=0$ and $N$ is active short then either:
  \begin{enumerate}[label=--]
  \item
$\rho_{0^-+1}^N=\rho_{m_{k-1}+1}^{M_{k-1}}$ and
$\rho_{0^-+1}^{N'}=\rho_{m_{k-1}+1}^{U_{k-1}}$,
  or
  \item
$\rho_{0^-+1}^N>\rho_{m_{k-1}+1}^{M_{k-1}}$ and
  $\rho_{0^-+1}^{N'}>\rho_{m_{k-1}+1}^{U_{k-1}}$.
  \end{enumerate}
\item If it is not the case that [$n=0$ and $N$ is active short] then either:
  \begin{enumerate}[label=--]
  \item
$\rho_n^N=\rho_{m_{k-1}+1}^{M_{k-1}}$
 and
   $\rho_n^{N'}=\rho_{m_{k-1}+1}^{U_{k-1}}$,
   or
  \item
$\rho_n^N>\rho_{m_{k-1}+1}^{M_{k-1}}$ and
  $\rho_n^{N'}>\rho_{m_{k-1}+1}^{U_{k-1}}$.
  \end{enumerate}
\end{enumerate}
\end{clm}
\begin{proof}
Suppose $M_{k-1}\pins N$. Then
$\rho\eqdef\rho_{m_{k-1}+1}^{M_{k-1}}=\rho_\om^{M_{k-1}}$
is an $N$-cardinal. We have $\rho_n^N\geq\rho$, because if $\rho_n^N<\rho$
then $n>0$, and letting $n'\in\om\cup\{0^-\}$ be least such that $\rho_{n'+1}^N<\rho$,
then $(N,n')$ should have been in the dropdown sequence, a contradiction. Similarly,
if $N$ is active short and $n=0$ then $\rho_{0^-+1}^N\geq\rho$, since otherwise $(N,0^-)$ should be in the dropdown sequence.

One now proceeds as in the case that $i+1<k$
to verify the rest of the claim.
\end{proof}

It just remains to deal with the case that $i=k-1$ and $M_{k-1}=N$.

\begin{clm}Suppose $M_{k-1}=N$, so $m_{k-1}<n$. Then $U_{k-1}\ins N'$ and  $\rho_{m_{k-1}+1}^{U_{k-1}}$ is an $N'$-cardinal. Clauses \ref{item:U_i,m_i_pins_U_i+1,m_i+1} and \ref{item:dropdown_of_consec_pair} hold for $i=k-1$, and  moreover:
 \begin{enumerate}[label=(\alph*)]
  \item If  $n>0$ then either:
  \begin{enumerate}[label=--]
   \item
 $\rho_n^N=\rho_{m_{k-1}+1}^N$  and
$\rho_n^{N'}=\rho_{m_{k-1}+1}^{U_{k-1}}$,
or
   \item
 $\rho_n^N>\rho_{m_{k-1}+1}^N$  and
$\rho_n^{N'}>\rho_{m_{k-1}+1}^{U_{k-1}}$.
\end{enumerate}

\item If $n=0$ then $m_{k-1}=0^-$, $M_{k-1}=N$ and
$U_{k-1}= N'$.
\end{enumerate}
\end{clm}
\begin{proof}
Suppose $n>0$. Then $\rho_n^N=\rho_{m_{k-1}+1}^N$, since if
 $\rho_n^N<\rho_{m_{k-1}+1}^N$ then again, there should have been another
element in the dropdown sequence.
Suppose instead $n=0$.
Then $m_{k-1}=0^-$,
so $U_{k-1}=U_k$,
since $\Ult_{0^-}$ and $\Ult_0$ are equivalent.

Using these observations, one proceeds as before to prove the claim.
\end{proof}
This completes the proof of the lemma.
\end{proof}

\section{Normalization}

We now work through normalization of iteration trees. The only portion of this in which there is really new content (over that in \cite{fullnorm}) is in the basic combinatorics of converting two sequences $\vec{G}$ and $\vec{F}$ of extenders, used along branches of iteration trees $\Tt$ and $\Uu$, where $(\Tt,\Uu)$ is a stack, into a single sequence $\vec{G}\cd\vec{F}$, used along a branch of the normalization of $(\Tt,\Uu)$.
We deal with this in \S\ref{sec:ult_comm}.

\subsection{Ultrapower commutativity}\label{sec:ult_comm}

The following lemma is just like \cite[Lemma 2.9]{fullnorm}, except that some of the extenders involved can be long. It deals with the normalization of the sequence $\vec{G}\cd\vec{F}$ as mentioned above, in which $\vec{G}$ only consists of a single extender $G$ (and in the notation in the lemma, instead of ``$\vec{F}$'', we work with $\vec{E}\conc\vec{F}$, so we are considering converting the pair $(\left<G\right>,\vec{E}\conc\vec{F})$ into the single sequence $\left<G\right>\cd(\vec{E}\conc\vec{F})$).

\begin{lem}\label{lem:extender_comm}
 See Figure \ref{fgr:extender_comm}, which is replicated from \cite{fullnorm}. Let $m\in\om\cup\{0^-\}$. Let $M,P$ be premice, with $M$ being $m$-standard
 and $P$ active.
 If $F^P$ is short, let $p=0^-$, and if $F^P$ is long, let $p=0$.
 Let
 $G=F^P$ and $\delta=\spc(G)$. Suppose $\delta^{+M}<\OR^M$,
 $M|\delta^{+M}=P|\delta^{+P}$
 and $\delta<\rho_m^M$.
 Suppose that if $M$ is active short and $m=0$ then also $\delta<\rho_{0^-+1}^M$. Let
$U=\Ult_m(M,G)$
and suppose $U$ is wellfounded. 
 
 Let $\vec{E}$ be $(M,m)$-good, $(P,p)$-good
 and $\delta$-space-bounded.
  Let
  \[ M_{\cd}=\Ult_m(M,\vec{E})\text{ and 
}P_\cd=\Ult_0(P,\vec{E})\text{ and }
  G_\cd=F^{P_\cd}, \]
so   
\[ \delta_{\cd}\eqdef 
i^{M,m}_{\vec{E}}(\delta)=i^{P,0}_{\vec{E}}(\delta)=\spc(G_\cd)<
\rho_m^{M_{\cd}}. \]
If $\vec{E}$ is $(U,m)$-pre-good, also let
$U_{\cd}=\Ult_m(U,\vec{E})$.

  Let $\vec{F}$ be $(P_\cd,p)$-good with
$\delta_{\cd}<\spc(F_\alpha)$ for each $\alpha<\lh(\vec{F})$.
Let $\vec{D}=\vec{E}\conc\vec{F}$. Let
\[ P^\cd=\Ult_0(P,\vec{E}\conc\vec{F})\text{ and }
G^\cd=F^{P^\cd}.\]
Let 
 $\delta^{\cd}=\spc(G^{\cd})=\spc(G_{\cd})=\delta_{
\cd}$.

If $\vec{E}\conc\vec{F}$ is $(U,m)$-pre-good, also let
$U^{\cd}=\Ult_m(U,\vec{E}\conc\vec{F})=\Ult_m(U_{\cd},\vec{F})$.

Let
 \[ \wt{U}_{\cd}=\Ult_m(M_{\cd},G_\cd)\text{ and 
}\wt{U}^{\cd}=\Ult_m(M_{\cd},G^\cd)\]
\tu{(}see part \ref{item:M',N'_agmt} below\tu{)},
and suppose $\wt{U}^{\cd}$ is 
wellfounded.  Then:
 \begin{enumerate}
   \item $\vec{E}\conc\vec{F}$ is $(U,m)$-good.
   \item\label{item:M',N'_agmt} $M_{\cd}||
   \delta_{\cd}^{+{M_{\cd}}}
=P_\cd|\delta_{\cd}^{+P_\cd}
=P^\cd|(\delta^{\cd})^{+P^\cd}$
   \tu{(}so $\wt{U}_{\cd}$ and $\wt{U}^{\cd}$
   are well-defined\tu{)},
  \item $U_{\cd}=\wt{U}_{\cd}$ and 
$U^\cd=\wt{U}^{\cd}$.
  \item\label{item:maps_commute} The various ultrapower maps commute,
  as indicated in Figure \ref{fgr:extender_comm}; that is,
  \[ i^{U,m}_{\vec{E}\conc\vec{F}}\com 
i^{M,m}_G=
i^{\wt{U}_{\cd},m}_{\vec{F}}\com 
i^{M_{\cd},m}_{G_\cd}\com i^{M,m}_{\vec{E}}=
i^{M_{\cd},m}_{G^\cd}\com 
i^{M,m}_{\vec{E}}.\]
   \item\label{item:ult_map_is_SL_map} The hypotheses for the Shift Lemma\footnote{That is, the relevant version of the Shift Lemma \cite[Shift Lemmas I--IV (***2.46--2.49)]{kappa-plus_v3}.} hold
with respect to $(M,P)$,
$(M_{\cd},P^\cd)$,
   and the maps
   $i^{M,m}_{\vec{E}}:M\to M_{\cd}$ and
    $i^{P,0}_{\vec{E}\conc\vec{F}}:P\to P^\cd$.
    Moreover, $i^{U,m}_{\vec{E}\conc\vec{F}}$ is just the Shift Lemma map.
 \end{enumerate}
\end{lem}

\begin{figure}
\centering
\begin{tikzpicture}
 [mymatrix/.style={
    matrix of math nodes,
    row sep=1.6cm,
    column sep=1.2cm}
  ]
   \matrix(m)[mymatrix]{
 U^{\cd}=\wt{U}^{\cd}&{}&P^\cd\\
 U_{\cd}=\wt{U}_{\cd} &M_{\cd}  &P_\cd   \\
U         &M       & P \\
};

\path[->,font=\scriptsize]
(m-3-1) edge node[left] {$\vec{E},m$} (m-2-1)
(m-2-1) edge node[left] {$\vec{F},m$} (m-1-1)
(m-3-2) edge node[below] {$G,m$} (m-3-1)
(m-3-2) edge node[left] {$\vec{E},m$} (m-2-2)
(m-2-2) edge node[below] {$G_\cd,m$} (m-2-1)
(m-2-2) edge node[right,pos=0.6] {$\ \ G^\cd,m$} (m-1-1)
(m-3-3) edge node[right] {$\vec{E},0$} (m-2-3)
(m-2-3) edge node[right] {$\vec{F},0$} (m-1-3)
(m-3-1) edge[bend left=65] node[left] {$\vec{D},m$} (m-1-1)
(m-3-3) edge[bend left] node[left] {$\vec{D},0$} (m-1-3)
;
\end{tikzpicture}
\caption{Extender commutativity. The diagrams commute,
where $\vec{D}=\vec{E}\conc\vec{F}$,
and a label $\vec{C},k$ denotes
a degree $k$ abstract iteration map given by $\vec{C}$.} 
\label{fgr:extender_comm}
\end{figure}
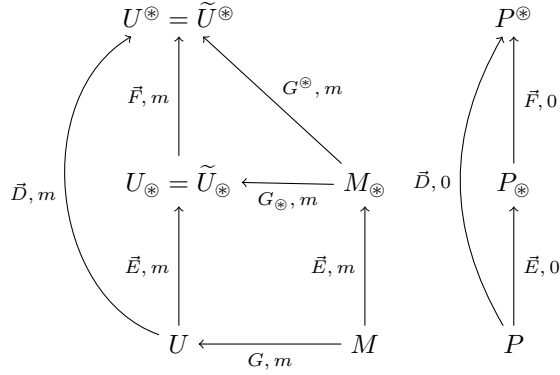
\begin{proof}
The proof is similar to that of \cite[Lemma 2.9]{fullnorm}, and essentially identical
unless
\begin{enumerate}[label=(\roman*)]
 \item\label{item:M_short_and_m=0,0-} $M$ is active short and $m\in\{0,0^-\}$
 and some ultrapower which is a component (via a single $E_\alpha$) of $\Ult_m(M,\vec{E})$ is formed avoiding a protomouse,
 or the ultrapower $\Ult(M,G)$
 is formed avoiding a protomouse,
 or
 \item\label{item:P_short} $P$ is active short and some ultrapower
 which is a component of $\Ult_0(P_{\cd},\vec{F})$
 is formed avoiding a protomouse
\end{enumerate}
 (note that because $\vec{E}$ is $\delta$-space bounded (where $\delta=\spc(F^P)$), no component of $\Ult_0(P,\vec{E})$ is formed avoiding a protomouse).
 
 \begin{case}Neither Case \ref{item:P_short}
 nor Case \ref{item:M_short_and_m=0,0-}
 attains.
 
 Then use the proof of \cite[Lemma 2.9]{fullnorm}.
 \end{case}

 \begin{case}\label{case:P_short_not_M_short_m=0,0-} Case \ref{item:P_short}
 attains, but Case \ref{item:M_short_and_m=0,0-}
 does not.
 
 Here the component of the lemma
 dealing with $\vec{E}$ is also
 as in \cite{fullnorm},
 so we may assume $\vec{E}=\emptyset$.
 Note that the extender derived from $i^{U,m}_{\vec{F}}$ is short with space $\crit(F^P)$,
 and that it is in fact identical
 to $F^{P^{\cd}}$; this uses the manner in which
 ultrapowers are formed when they avoid a protomouse. Using this observation, the rest of the proof is as in
  \cite{fullnorm}, except for part \ref{item:ult_map_is_SL_map}: here we don't have full agreement of $i^{M,m}_{\vec{E}}$
  with $i^{P,0}_{\vec{E}\conc\vec{F}}$
  over $\kappa^{+P}$ where $\kappa=\crit(F^P)$,
  since in fact $i^{M,m}_{\vec{E}}(\kappa)<i^{P,0}_{\vec{E}\conc\vec{F}}(\kappa)$. (Moreover, with our simplifying assumption that $\vec{E}=\emptyset$, $i^{M,m}_{\vec{E}}=\id$, but $\crit(i^{P,0}_{\vec{E}\conc\vec{F}})=\kappa$.)
  But note that the hypotheses of Shift Lemma IV \cite[***2.49]{kappa-plus_v3} hold, and $i^{U,m}_{\vec{E}\conc\vec{F}}$ is just the resulting map (and part \ref{item:ult_map_is_SL_map}
  refers to this version of the Shift Lemma in this case).
 \end{case}

 It remains to deal with  Case \ref{item:M_short_and_m=0,0-}.
 \begin{case} Case \ref{item:M_short_and_m=0,0-} attains, $F^P$ is long and $\crit(F^M)=\crit(F^P)$.

 Note that the previous cases already establish the instance of the lemma given by replacing $M$ with $M^{\passive}$ (and hence $U$ with $U^{\passive}$, etc).
So the only things we need to verify here are that $F^{U_{\cd}}=F^{\widetilde{U}_{\cd}}$ and $F^{U^{\cd}}=F^{\widetilde{U}^{\cd}}$,
and the relevant version of the Shift Lemma applies.\footnote{That is,
we already know that the Shift Lemma applies if we replace $M$ with $M^{\passive}$, etc, but we must now also consider the elementarity of the maps with respect to the active extenders.}

 Since $F^P$ is long, Case \ref{item:P_short}
 does not attain.
 
 Note that because $\Ult_0(M,\vec{E})$ is wellfounded, there are only finitely many $\alpha<\lh(\vec{E})$
 such that $E_\alpha$ is long with $\crit(E_\alpha)=i^{M,0}_{\vec{E}\rest\alpha}(\crit(F^M))$. Let these be $\alpha_0<\alpha_1<\ldots<\alpha_{n-1}$, and suppose for now that
 $n>0$.
 
 We have $U=\Ult_0(M,G)$, which is formed to avoid a protomouse.
 Let $\eta=\sup i^{M,0}_{G}``\delta$
 (recall $\delta=\kappa^{+M}$ where $\kappa=\crit(F^M)=\crit(F^P)$).
 So \begin{equation}\label{eqn:F^U=}F^U=F^{U^*}\com F^{U|\eta},\end{equation}
 where $U^*$ is the protomouse associated to $U$,
 and so
 \begin{equation}\label{eqn:F^U*=}F^{U^*}=\bigcup i^{M,0}_G``F^M\end{equation}
 We have that $F^U$ is short with $\kappa=\crit(F^U)$
 and $\delta=\kappa^{+U}$.
 
 For $\alpha\leq\lh(\vec{E})$
 write $U_\alpha=\Ult_0(U,\vec{E}\rest\alpha)$
 and $j_{\alpha\beta}:U_\alpha\to U_\beta$ for the ultrapower map (note $U_0=U$).
 Now $j_{0\alpha_0}$
 is continuous at $\delta$ and at $\eta$,
 and
 \begin{equation}\label{eqn:F^U_alpha_0}F^{U_{\alpha_0}}=\bigcup j_{0\alpha_0}``F^U.\end{equation}
 Now the ultrapower $\Ult_0(U_{\alpha_0},E_{\alpha_0})$ avoids a protomouse, so
 $j_{\alpha_0,\alpha_0+1}$ is discontinuous
 at $j_{0\alpha_0}(\delta)$ and
 so \[ F^{U_{\alpha_0+1}}=F^{U_{\alpha_0+1}^*}\com F^{U_{\alpha_0+1}|\eta'},\]
where $\eta'=\sup j_{\alpha_0,\alpha_0+1}``j_{0\alpha_0}(\delta)$ and $U_{\alpha_0+1}^*$
 is the protomouse associated to $U_{\alpha_0+1}$
 (or more precisely, associated to the formation of this as the ultrapower of $U_{\alpha_0}$ by $E_{\alpha_0}$; we do not mean to imply that the protomouse is uniquely determined by the structure $U_{\alpha_0+1}$), so
 \[ F^{U_{\alpha_0+1}^*}=\bigcup j_{\alpha_0,\alpha_0+1}``F^{U_{\alpha_0}}.\]
 
 But putting these things together, we get that
 \[ F^{U_{\alpha_0+1}}=F^*\com F^{U_{\alpha_0+1}|\eta''},\]
 where $\eta''=\sup j_{0,\alpha_0+1}``\eta$
 and $F^*=\bigcup j_{0,\alpha_0+1}``F^{U^*}$;
 this is because 
 \[ i^{F^{U_{\alpha_0+1}|j_{0,\alpha_0+1}(\eta)}}(\eta')=\eta'' \]
 and \[i^{F^{U_{\alpha_0+1}|j_{0,\alpha_0+1}(\eta)}}(F^{U_{\alpha_0+1}|\eta'})=F^{U_{\alpha_0+1}|j_{0,\alpha_0+1}(\eta)}\com F^{U_{\alpha_0+1}|\eta'}.\]

Proceeding in this manner through $n$ steps,
and letting $j:M\to U_{\cd}$ be $j=i^{U,0}_{\vec{E}}\com i^{M,0}_G$, we get that
\begin{equation}\label{eqn:F^U_cd=} F^{U_{\cd}}=F^{**}\com F^{U_{\cd}|\gamma},\end{equation}
where
$F^{**}=\bigcup j``F^M$
and
$\gamma=\sup j``\delta$.

Similarly (but slightly more simply),
letting $i:M\to\Ult_0(M,\vec{E})$
be the ultrapower map, we have
\[ F^{M_{\cd}}=F'\com F^{M_{\cd}|\xi},\]
where $F'=\bigcup i``F^M$
and $\xi=\sup i``\delta$.

Now as remarked at the start of this case,
we already know that $(\widetilde{U}_{\cd})^{\passive}=(U_{\cd})^{\passive}$
and the diagram commutes; that is,
$\Ult_0(U^{\passive},\vec{E})=\Ult_0((M_{\cd})^{\passive},F^{P_{\cd}})$
and $k\com i=j$ where $k:M_{\cd}\to\Ult_0(M_{\cd},F^{P_{\cd}})$ is the ultrapower map. So \begin{equation}\label{eqn:proto_exts_agree_U,U-tilde}\bigcup k``F'=\bigcup k\com i``F^M=\bigcup j``F^M=F^{**}\end{equation}
and $\gamma=\sup j``\delta=\sup k\com i``\delta=\sup k``\xi$.

Now $\crit(F^{P_{\cd}})=i^{P,0}_{\vec{E}}(\kappa)>\crit(F^{M_{\cd}})$ (as $n>0$).
So $\widetilde{U}_{\cd}=\Ult_0(M_{\cd},F^{P_{\cd}})$
is formed in the usual manner,
and so
\[ F^{\widetilde{U}_\cd}=\bigcup k``F^{M_{\cd}}=\Big(\bigcup k``F'\Big)\com\bigcup k``F^{M_\cd|\xi},\]
but $k$ is continuous at $\xi$
and $k(\xi)=\gamma$, so $\bigcup k``F^{M_\cd|\xi}=F^{U_{\cd}|\gamma}$,
which by lines (\ref{eqn:F^U_cd=}) and (\ref{eqn:proto_exts_agree_U,U-tilde})  gives $F^{\widetilde{U}_{\cd}}=F^{U_{\cd}}$, as desired.
Finally for part \ref{item:ult_map_is_SL_map}, note that the hypotheses of Shift Lemma II hold (\cite[***Lemma 2.47]{kappa-plus_v3}), using its clause (b), and $i^{U,m}_{\vec{E}\conc\vec{F}}$ is just the resulting map.

Now suppose  instead that $n=0$. Then $i$ is continuous at $\delta$ and
\[ F^{M_{\cd}}=\bigcup i``F^M,\]
and in particular, $\crit(F^{M_{\cd}})=i(\kappa)=\crit(F^{P_{\cd}})$,
so $\Ult_0(M_{\cd},F^{P_{\cd}})$ is formed avoiding the protomouse (as is $U=\Ult_0(M,G)$).
A simplification of the $n>0$ calculation gives  $F^{U_{\cd}}=\big(\bigcup j``F^M\big)\com F^{U_{\cd}|\gamma}$
where 
 $\gamma=\sup j``\delta=\sup i^{U,0}_{\vec{E}}``\eta=i^{U,0}_{\vec{E}}(\eta)$,
and that $F^{\widetilde{U}_{\cd}}=\big(\bigcup k``F^{M_{\cd}}\big)\com F^{U_{\cd}|\gamma'}$,
where $\gamma'=\sup k``i(\delta)$.
But  $\gamma=\gamma'$,
and since $j=k\com i$,
this gives  $F^{U_{\cd}}=F^{\widetilde{U}_{\cd}}$.

So we have shown that $F^{U_{\cd}}=F^{\widetilde{U}_{\cd}}$. We now want to see that $F^{U^{\cd}}=F^{\widetilde{U}^{\cd}}$. But  because we have already established everything regarding $\vec{E}$,
and in the present case, 
Case \ref{item:P_short} does not attain,
the calculation involving $\vec{F}$ is just as in \cite{fullnorm}.
And for part \ref{item:ult_map_is_SL_map}, we use Shift Lemma I  (\cite[***Lemma 2.46]{kappa-plus_v3}),
clause (b).
 \end{case}

 \begin{case} Case \ref{item:M_short_and_m=0,0-} attains and 
  if $G=F^P$ is long then $\crit(F^M)\neq\crit(F^P)$.

Again
we just need to verify here  that $F^{U_{\cd}}=F^{\widetilde{U}_{\cd}}$ and $F^{U^{\cd}}=F^{\widetilde{U}^{\cd}}$ and the Shift Lemma applies appropriately.

 We first claim $\crit(F^M)<\crit(F^P)$.
For clearly $\crit(F^M)\leq\crit(F^P)$.
Suppose
 $\crit(F^M)=\crit(F^P)$. Then by case hypothesis, $F^P$ is short, so $\delta=\kappa=\crit(F^M)=\crit(F^P)$.
 Then since some $\Ult_0(M_\alpha,E_\alpha)$ is formed avoiding the protomouse, note that $\vec{E}$ is not $\delta$-space-bounded, a contradiction.
 
 Now both $\Ult_0(M,F^P)$
 and $\Ult_0(M_{\cd},F^{P_{\cd}})$
 are formed in the usual manner,
 so we only have the (finitely many)
 component ultrapowers $\Ult_0(M_{\alpha_i},E_{\alpha_i})$ of $\Ult_0(M,\vec{E})$ formed avoiding a protomouse,
 say for $i<n$, where $n<\om$.
 We have $\crit(F^U)=\crit(F^M)<\crit(F^P)$
 and $F^U$ is short,
 and so $\Ult_0(U_{\alpha_i},E_{\alpha_i})$
 is also formed avoiding a protomouse
 for $i<n$,
 and only these components of $\Ult_0(U,\vec{E})$
 are formed in this manner. But now a calcluation
 simpler than the previous case gives that $F^{U_{\cd}}=F^{\widetilde{U}_{\cd}}$, as desired.

 So we have $U_{\cd}=\widetilde{U}_{\cd}$,
 completing the calculation for $\vec{E}$.
 As before, the calculation for $\vec{F}$
 is  as in \cite{fullnorm}
 unless Case \ref{item:P_short} attains, so suppose it does. So $F^P$ is short,
 as is $F^{P_{\cd}}$,
 and some component of $\Ult_0(P_{\cd},\vec{F})$
 is formed avoiding a protomouse. But then the calculation for $\vec{F}$ is just like that in Case \ref{case:P_short_not_M_short_m=0,0-} (slightly generalized to allow $\vec{E}\neq\emptyset$). For part \ref{item:ult_map_is_SL_map},  use Shift Lemma II  (\cite[***2.47]{kappa-plus_v3}).\qedhere
 \end{case}
\end{proof}

As in \cite[***2.11, 2.12]{fullnorm},
we next want to generalize the preceding lemmas to deal with the case
of a (normal) sequence $\vec{G}$ of extenders, hence producing the normalization $\vec{G}\cd\vec{F}$ of the pair $(\vec{G},\vec{F})$.

\begin{dfn}
 Let $\vec{P}=\left<P_\alpha\right>_{\alpha<\lambda}$ be a sequence of active
premice.
We say $\vec{P}$ and $\left<F^{P_\alpha}\right>_{\alpha<\lambda}$
 are \emph{normal} iff $\lgcd(P_\alpha)\leq\spc(F^{P_\beta})$
 and $(P_\alpha)^\passive\pins_{\card} P_\beta$
 for $\alpha<\beta<\lambda$.\footnote{Note that the requirement that $\lgcd(P_\alpha)\leq\spc(F^{P_\beta})$ matches
 the rules for determining tree order in $k$-maximal trees specified in \cite[***Definition 2.38]{kappa-plus_v3}.}
\end{dfn}

\begin{dfn}\label{dfn:G*F}
 Let $\left<P_\alpha\right>_{\alpha<\theta}$
 be a normal sequence of active premice.
  Let $F_\alpha=F^{P_\alpha}$,
 and $\vec{F}=\left<F_\alpha\right>_{\alpha<\theta}$.
  Let $\left<Q_\alpha,G_\alpha\right>_{\alpha<\lambda}$ and $\vec{G}$ be
likewise. If $G_\alpha$ is short, let $q_\alpha=0^-$,
and otherwise let $q_\alpha=0$.
 
 Let $\alpha<\lambda$. Let $\eta_\alpha$ be the largest $\eta\leq\theta$
 such that $\vec{F}\rest\eta$ is
 $(Q_\alpha,q_\alpha)$-pre-good.\footnote{Recall that this includes the demand that $\vec{F}\rest\eta$ is   ${<\lgcd(Q_\alpha)}$-bounded.}
 Suppose that $\vec{F}\rest\eta_\alpha$ is $(Q_\alpha,q_\alpha)$-good.
 Let $\xi_\alpha$ be the largest $\xi\leq\theta$ such that
 $\vec{F}\rest\xi$ is $(Q_\alpha,q_\alpha)$-pre-good
 and $\spc(F^{Q_\alpha})$-bounded.
 Note that $\xi_\alpha\leq\eta_\alpha$,
 so $\vec{F}\rest\xi_\alpha$ is also $(Q_\alpha,q_\alpha)$-good.
By normality, $\eta_\beta\leq\xi_\alpha$ for $\beta<\alpha$.

 Write $Q^\cd_\alpha=\Ult_0(Q_\alpha,\vec{F}\rest\eta_\alpha)$ and 
$G^\cd_\alpha=F^{Q^\cd_\alpha}$.
 Given $\beta<\theta$, say that $F_\beta$ is \emph{nested} (with respect to this $\cd$-product) iff 
$\xi_\alpha\leq\beta<\eta_\alpha$ for some $\alpha<\theta$;
 and \emph{unnested} otherwise.
  Then the $\cd$-product $\vec{G}\cd\vec{F}$ denotes the enumeration of
 \[ X=\{G^\cd_\alpha\}_{\alpha<\lambda}\cup\{F_\alpha\mid\alpha<\theta\text{ and
}F_\alpha\text{ is unnested}\} \]
 in order of increasing critical point.
 And $\vec{Q}\cd\vec{P}$ denotes the corresponding enumeration of
 \[ \{Q^{\cd *}_\alpha\}_{\alpha<\lambda}\cup\{P_\alpha\mid\alpha<\theta\text{ and 
}F_\alpha\text{ is unnested}\}.\]
 
 Also write
$Q_{\alpha\cd}=\Ult_0(Q_\alpha,\vec{F}\rest\xi_\alpha)$
 and $G_{\alpha\cd}=F^{Q_{\alpha\cd}}$.
\end{dfn}
\begin{lem}
 Adopt the hypotheses and notation of \ref{dfn:G*F}.  Let $m\in\om\cup\{0^-\}$. Let $M$ be an $m$-standard premouse.
 Suppose $\vec{G}$ is $(M,m)$-pre-good and
 $\vec{G}\cd\vec{F}$ is $(M,m)$-good.
 Let $U=\Ult_m(M,\vec{G})$.
Then:
 \begin{enumerate}
  \item\label{item:crits_disagree} If $E,F\in X$ then $\crit(E)\neq\crit(F)$, so 
the ordering of $\vec{G}\cd\vec{F}$ is well-defined.
  \item\label{item:seqs_are_normal} $\vec{G}\cd\vec{F}$ and $\vec{Q}\cd\vec{P}$ 
are normal sequences.
  \item\label{item:*_matches_comp} $\vec{G}\cd\vec{F}$ is equivalent to 
$\vec{G}\conc\vec{F}$ with respect to $(M,m)$; that is,
 $\vec{G}$ is $(M,m)$-good, $\vec{F}$ is $(U,m)$-good,
  \[ \Ult_m(M,\vec{G}\conc\vec{F})=\Ult_m(M,\vec{G}\cd\vec{F}) \]
  and the associated ultrapower maps \tu{(}and hence derived extenders\tu{)} 
agree.
 \end{enumerate}
\end{lem}

The proof is essentially
the same as that of \cite[***Lemma 2.12]{fullnorm},
but using  Lemma \ref{lem:extender_comm}
for the successor steps. We leave the verification to the reader.

\subsection{Minimal strategy condensation}\label{sec:min_strat_cond}

The remaining details of normalization are very close to those in \cite{fullnorm},
incorporating some simple modifications relating to the modified rules for determining tree order in $n$-maximal trees, and the adjustments to degrees, dealing with degree $0^-$ and $0$ appropriately, and replacing ``drops in model or degree'' with ``drops of any kind''. In \S\S\ref{sec:min_strat_cond}--\ref{sec:comparison_analysis} we enumerate the points where such modifications should be made.

We define the \emph{dropdown sequence} $\dds^{(\Tt,\theta)}$ for a putative tree $\Tt$ and $\theta\leq\lh(\Tt)$, much as in \cite[Definition 3.1***]{fullnorm}, but incorporating the degree $0^-$ for active short levels:

\begin{dfn}[Tree dropdown]\label{dfn:tree_dropdown}
Let $m\in\om\cup\{0^-\}$, let $M$ be $m$-standard,
and let $\Tt$ be a putative $m$-maximal tree on $M$.

Let $\beta<\lh(\Tt)$. Then
 $(P_\beta,d_\beta)$ denotes
\begin{enumerate}[label=--]
 \item $(\exit^\Tt_\beta,0^-)$, if $\beta+1<\lh(\Tt)$ and $E^\Tt_\beta$ is short,\footnote{In \cite{fullnorm}, instead of defining
a direct analogue of $P_\beta$, its ordinal height is defined and denoted $\lambda_\beta$.}
\item $(\exit^\Tt_\beta,0)$, if $\beta+1<\lh(\Tt)$ and $E^\Tt_\beta$ is long, and
 \item $(M^\Tt_\beta,\deg^\Tt_\beta)$,
 if $\beta+1=\lh(\Tt)$ and $M^\Tt_\beta$ well-defined.
\end{enumerate}
Let 
$\left<M_{\beta i},m_{\beta i}\right>_{i\leq k_\beta}$ be the 
\emph{reversed} extended dropdown of
\[ ((M^\Tt_\beta,\deg^\Tt_\beta),(P_\beta,d_\beta))\] (note 
this defines $k_\beta$).
Then  $k_\beta^{\Tt}\eqdef k_\beta$ and $M_{\beta i}^{\Tt}\eqdef M_{\beta i}$ and $m^\Tt_{\beta i}=m_{\beta i}$. Let $\theta\leq\lh(\Tt)$.
We define the \emph{dropdown domain} $\ddd^{(\Tt,\theta)}$ of $(\Tt,\theta)$ by
\[ \Delta=\ddd^{(\Tt,\theta)}\eqdef\{(\beta,i)\mid\beta<\theta\ \&\ i\leq k_\beta\}, \]
and define the \emph{dropdown 
sequence} $\dds^{(\Tt,\theta)}$ of $(\Tt,\theta)$ by
\[ \dds^{(\Tt,\theta)}\eqdef\left<(M_{\beta 
i},m_{\beta i})\right>_{(\beta,i)\in\Delta}. \]

The \emph{dropdown sequence} $\dds^\Tt$ of $\Tt$ is 
$\dds^{(\Tt,\lh(\Tt))}$, and 
the \emph{dropdown domain} $\ddd^\Tt$ of $\Tt$ is $\ddd^{(\Tt,\lh(\Tt))}$.

Given $\varsigma<\lgcd(\exit^\Tt_\alpha)$
for some $\alpha+1<\lh(\Tt)$,
$\alpha^\Tt_\varsigma$ denotes the least such $\alpha$,
and $n^\Tt_\varsigma$ denotes the largest $n\leq k^\Tt_\alpha$
such that $n=0$ or
$\rho_{m_{\alpha n}+1}(M_{\alpha i})\leq\varsigma$.
If instead $\lh(\Tt)=\alpha+1$
and $\varsigma\leq\OR(M^\Tt_\alpha)$
but $\lgcd(\exit^\Tt_\beta)\leq\varsigma$
for all $\beta+1<\lh(\Tt)$, then $\alpha^\Tt_\varsigma$ denotes $\alpha$
and $n^\Tt_\varsigma$ denotes $0$.
\end{dfn}

So if  $\beta+1<\lh(\Tt)$ and
 $\varsigma=\spc(E^\Tt_\beta)$ then
 $\pred^\Tt(\beta+1)=\alpha^\Tt_\varsigma$
 and 
$M^{*\Tt}_{\beta+1}=M^\Tt_{\alpha^\Tt_\varsigma n^\Tt_\varsigma}$.

\newcommand{\pre}{\mathrm{pre}}

\begin{dfn}[Tree pre-embedding]\label{dfn:tree_embedding}
(Cf.~\cite[Figure 1]{iter_for_stacks}.)
Let $M$ be an $m$-standard premouse, let $\Tt,\Xx$ be putative
$m$-maximal 
trees on $M$,
with $\Xx$ a true tree,
and $\theta\leq\lh(\Tt)$. A \emph{tree pre-embedding} from $(\Tt,\theta)$ 
to $\Xx$,
denoted
\[ \Pi:(\Tt,\theta)\hookrightarrow_{\pre}\Xx, \]
is a sequence $\Pi=\left<I_\alpha\right>_{\alpha<\theta}$
with properties as in \cite[Definition 3.3***]{fullnorm}, replacing ``$\deg$'' with ``${\D}$-$\deg$''
throughout (in its clauses 5 and 7). As there, we say that $\Pi$ has \emph{degree $m$}.
\end{dfn}

\begin{rem}\label{dfn:F-vec_tree_emb}\label{lem:intervals_I_cover_X-branches}
 \cite[Remark 3.4]{fullnorm} continues to hold
(after replacing ``$\deg$''
with ``$\D$-$\deg$'').
And  $\vec{E}^{\Xx}_{\alpha\beta}$
was defined in \cite[Definition 3.5]{fullnorm}.
Given $\Pi:(\Tt,\theta)\hookrightarrow_{\pre}\Xx$,
the \emph{inflationary extender sequences} $\vec{F}_\xi=\vec{F}_\xi^{\Pi}$ are defined just as in \cite[Definition 3.6]{fullnorm}.
\end{rem}

We next define \emph{minimal tree embedding}
just as in \cite{fullnorm}:
\begin{dfn}[Minimal tree embedding]\label{dfn:min_tree_emb}
Let $\Pi:(\Tt,\theta)\hookrightarrow_{\pre}\Xx$  be a tree pre-embedding.
We say $\Pi$ is 
\emph{minimal}, denoted 
\[ \Pi:(\Tt,\theta)\hookrightarrow_\minim\Xx,\]
provided the properties in \cite[***Definition 3.7]{fullnorm} hold,
after replacing ``drop in model or degree''
with ``drop in any way (that is, model, degree or Dodd-degree)''.
We define  \emph{bounding} and \emph{exactly bounding}  as in \cite[***3.7]{fullnorm}.
Also as in \cite{fullnorm}, 
we write $\Pi:\Tt\hookrightarrow_{\minim}\Xx$
iff $\Pi:(\Tt,\lh(\Tt))\tembto_{\minim}\Xx$.
\end{dfn}

Note that if $\Pi$ is a minimal tree pre-embedding,
then with $\alpha+1<\lh(\Tt)$
as in  \cite[***Definition 3.7, clause 1]{fullnorm}, for all $\beta+1\in I_\alpha\cut\{\gamma_\alpha\}$,
letting $\xi=\pred^\Xx(\beta+1)$,
we have $\spc(E^\Xx_\beta)<\lgcd(Q_{\alpha\xi})$,
because $\vec{F}_{\delta_\alpha}$
is $(\exit^\Tt_\alpha,0)$-good,
and hence $E^\Xx_\beta$ is $(Q_{\alpha\xi},0)$-good.

\begin{dfn}\label{dfn:min_dagger_tree_emb}\label{dfn:min_inf_ults_maps}
We define \emph{puta-minimal} tree pre-embedding
as in \cite[*** 3.8]{fullnorm}.

Let  $m\in\om\cup\{0^-\}$ and
$\Pi:(\Tt,\theta)\hookrightarrow_{\putamin}\Xx$
of degree $m$ on an $m$-standard $M$. Then we adopt the same notation introduced in \cite[***Definition 3.9]{fullnorm} dealing with dropdown lifts, and define \emph{pre-standard} as there.
(Note that although in \cite{fullnorm}, $m_{\beta k_\beta}=0$, here we have $m_{\beta k_\beta}\in \{0,0^-\}$, and $m_{\beta k_\beta}=0^-$ is possible.) \cite[***Remark 3.10]{fullnorm} continues to hold.
\end{dfn}

We next adapt the definition \emph{standard} from \cite[***3.11]{fullnorm}:

\begin{dfn}\label{dfn:min_inf_coherent}
Let $M,m,\Tt,\Xx,\Pi,\Delta$ be as in \cite[***3.9]{fullnorm}, as adpated in Definition  \ref{dfn:min_inf_ults_maps}.
We say $\Pi$ is \emph{standard} iff
\begin{enumerate}[label=T\arabic*.,ref=T\arabic*]
 \item\label{item:pre-standardness}  
$\Pi:(\Tt,\theta)\hookrightarrow_{\min}\Xx$ and $\Pi$ is pre-standard.  
\item\label{item:T_dropdowns_lift} (Dropdowns lift)
\cite[Definition 3.11(T2)]{fullnorm} holds,
after replacing $\mathscr{D}^\Xx_{\deg}$ with $\mathscr{D}^{\Xx}_{\D\text{-}\deg}$ in clauses (d) and (e),
and replacing clause (f) with the following:
\begin{enumerate}
\setcounter{enumii}{5}
\item\label{item:dropdown_correspondence} Suppose $\alpha+1<\lh(\Tt)$.
If $E^\Tt_\alpha$ is short, let $\ell=0^-$, and if long, let $\ell=0$.
Then:
\begin{enumerate}
 \item\label{item:dropdown_correspondence_gamma_alpha} $\left<(U_{\alpha j},m_{\alpha j})\right>_{j\leq k_\alpha}$ is
 the revex
$((M^\Xx_{\gamma_\alpha},\deg^\Xx_{\gamma_\alpha}),(Q_{\alpha\gamma_\alpha},
\ell))$-dd.
 \item If $\gamma_{\alpha i}<\xi\leq\delta_{\alpha i}$
 then $\left<(U_{\alpha j\xi},m_{\alpha j})\right>_{i\leq j\leq k_\alpha}$ is
 the revex
 $((M^\Xx_\xi,\deg^\Xx_\xi),(Q_{\alpha\xi},\ell))$-dd.
\end{enumerate}
\end{enumerate}

\item\label{item:T_emb_agmt} (Embedding agreement) Let $\alpha<\theta$ with $\alpha+1<\lh(\Tt)$
and  $\varsigma<\lgcd(\exit^\Tt_\alpha)$ be an $\exit^\Tt_\alpha$-cardinal
with $\alpha=\alpha^\Tt_\varsigma$ and
$i=n^\Tt_{\varsigma}$.
Let $\xi\in J_{\alpha i}$ be least such that,
letting $\mu=\pi_{\alpha i\xi}(\varsigma)$,
either $\xi=\delta_{\alpha i}$ or
$\mu<\spc(E^\Xx_\eta)$
where $\eta+1=\successor^\Xx(\xi,\delta_{\alpha i})$; note that
$\xi\geq\gamma_{\alpha i}$. Write $\xi_{\varsigma}=\xi$. 
Let $U=U_{\alpha i\xi}$ and $\pi=\pi_{\alpha i\xi}$.

Whenever $(\alpha',i',\xi')\geq(\alpha,i,\xi)$,
 $U'=U_{\alpha'i'\xi'}$ and $\pi'=\pi_{\alpha'i'\xi'}$, we have:
\begin{enumerate}[label=--]
\item $U||\mu^{+U}=
U'||\mu^{+U'}$,
\item either:
\begin{enumerate}[label=--]
 \item $\pi\rest\pow(\varsigma)\sub\pi'$, or
 \item  $\varsigma$ is inaccessible in  $M_{\alpha i}$ and if $(\alpha,i,\xi)<(\alpha',i',\xi')$
 then $\pi(\varsigma)<\pi'(\varsigma)$
 and $\pi(X)=\pi'(X)\cap\pi(\varsigma)$ for all $X\in\pow(\varsigma)\cap M_{\alpha i}$.
\end{enumerate}

\item if $\alpha<\alpha'$ and $(i,\xi)=(k^\Tt_\alpha,\delta_\alpha)$
then letting $\lambda=\lambda(E^\Tt_\alpha)$,\footnote{Recall here that if $E^\Tt_\alpha$
is long with $\nu(E^\Tt_\alpha)$ a limit,
then $\lambda^{+\exit^\Tt_\alpha}=\nu(E^\Tt_\alpha)=\lgcd(\exit^\Tt_\alpha)$,
which is the exchange ordinal associated to $E^\Tt_\alpha$.} we have:
\begin{enumerate}[label=--]
\item $\pi\rest\lambda\sub\pi'\rest\lambda$
and $\pi(\lambda)\leq\pi'(\lambda)$,
\item $\pi(X)=\pi'(X)\cap\pi(\lambda)$
for all $X\in\pow(\lambda)\cap\exit^\Tt_\alpha$,

\item 
$\lh(E^\Xx_{\delta_\alpha})\leq\pi'(\lh(E^\Tt_\alpha))$.
\end{enumerate}
\end{enumerate}

\item\label{item:T_commutativity} (Commutativity)
Condition \cite[Definition 3.11(T4)]{fullnorm} holds, after replacing $\mathscr{D}^{\Tt}_{\deg}$ and $\mathscr{D}^{\Xx}_{\deg}$ with $\mathscr{D}^{\Tt}_{\D\mathrm{-}\deg}$ and $\mathscr{D}^{\Xx}_{D\mathrm{-}\deg}$ throughout, and replacing clause (d) with:
\begin{enumerate}[label=\tu{(}\alph*\tu{)}]
\setcounter{enumii}{3}
\item\label{item:T_shift_lemma} (Shift Lemma)
Let $\varsigma=\spc(E^\Tt_\beta)$,
so $\varsigma<\lgcd(\exit^\Tt_\chi)$ and $i=n^\Tt_{\varsigma}$, so 
\ref{item:T_emb_agmt} applies. Then
\begin{enumerate}[label=(\roman*)]
\item $\xi$ (recall $\xi=\pred^\Xx(\gamma_{\beta+1})$)
is also as defined in \ref{item:T_emb_agmt} and
\[ (M^{*\Xx}_{\gamma_{\beta+1}},\deg^\Xx_{\gamma_{\beta+1}})
=(U_{\chi i\xi},m_{\chi i}).\]
So by  \ref{item:T_emb_agmt},
the Shift Lemma (see \cite[***2.46--2.49]{kappa-plus_v3}) applies to the embeddings
$\pi_{\chi i\xi}$
and
$\omega_{\beta\delta_\beta}:\exit^\Tt_\beta\to Q_{\beta\delta_\beta}$.\footnote{Note we might have $\omega_{\beta\delta_\beta}(\crit(E^\Tt_\beta))>\crit(Q_{\beta\delta_\beta})=\pi_{\chi i\xi}(\crit(E^\Tt_\beta))$,
in which case $\om_{\beta\delta_\beta}$ is $0$-deriving between active short premice,
and if $m_{\chi i}\in\{0,0^-\}$
then $\pi_{\chi i\xi}:M^\Tt_{\chi i}\to U_{\chi i\xi}$ can fail to be a $0$-embedding, in which case  is $0$-deriving between active short premice.}
\item 
$\pi_{\beta+1,0}$ is just the map given by the Shift Lemma
(this makes sense as $M^\Xx_{\gamma_{\beta+1}}=U_{\beta+1,0}$ by  
\ref{item:T_dropdowns_lift}).\qedhere
\end{enumerate}
\end{enumerate}
\end{enumerate}
\end{dfn}

The following lemma is proved by a straightforward
adaptation of \cite[***Lemma 3.12]{fullnorm}, which we leave to the reader:
\begin{lem}\label{lem:Pi_is_standard}
Let
$\Pi:(\Tt,\theta)\hookrightarrow_{\putamin}\Xx$ have degree $m$,
on an $m$-standard $M$.
Then $\Pi:(\Tt,\theta)\hookrightarrow_{\min}\Xx$, $\Pi$ is standard
and $\Tt\rest\theta$ has well-defined and wellfounded models.
\end{lem}

\begin{dfn}\label{dfn:Pi_subscript_notation}\label{dfn:pi_kappa}\label{dfn:trivial_tree_embedding}
For minimal tree embeddings $\Pi$, we use
notation analogous to that of \cite{iter_for_stacks} and \cite[***3.14]{fullnorm}; the subscript ``$\Pi$''
indicates objects associated to $\Pi$.

 Let
$\Pi:(\Tt,\theta)\tembto_{\minim}\Xx$
and $\gamma_\beta=\gamma_{\Pi\beta}$, 
etc. 
Let $\beta<\theta$
and $\varsigma\leq\OR(M^\Tt_\beta)$ with
$\beta=\alpha^\Tt_\varsigma$,
and let $n=n^\Tt_\varsigma$ (Definition \ref{dfn:tree_dropdown}).
Let  $N^\Tt_{\varsigma}=M^\Tt_{\beta n}$ and
$\xi=\xi_{\Pi\varsigma}\in I_{\beta n}$ be defined
as $\xi_\varsigma$ in
\ref{dfn:min_inf_coherent}(\ref{item:T_emb_agmt}),
or if $\lh(\Tt)=\beta+1$ and $\varsigma=\OR(M^\Tt_\beta)$,
then $\xi=\xi_{\Pi\varsigma}=\delta_\beta$.
Also let $U_{\Pi\varsigma}=U_{\beta n\xi}$ and
$\pi_{\Pi\varsigma}:N^\Tt_{\varsigma}\to U_{\Pi\varsigma}$
the corresponding ultrapower map.

Given $\gamma\in\theta\cap\lh(\Tt)^-$
and $\xi\in I_{\Pi\gamma}$,  $E_{\Pi\xi}$ denotes
$F^{Q_{\gamma\xi}}$, as in \cite[***3.16]{fullnorm}.

We define \emph{trivial}  and
 \emph{identity} tree embeddings as in \cite[***3.17]{fullnorm}.
\end{dfn}

\begin{rem}\label{lem:Pi_min_copy_normal}\label{dfn:one-step}\label{lem:one-step_copy_extension_is_standard}As in \cite{iter_for_stacks} and \cite{fullnorm},
we
can propagate minimal tree embeddings. Firstly, \cite[***Lemma 3.18]{fullnorm} continues to hold, via a very similar proof. Thus, given
$\Pi:(\Tt,\alpha+1)\hookrightarrow_{\min}\Xx$ of degree $m$ on an $m$-standard $M$,
where $\alpha+1<\lh(\Tt)$,
we can define the
\emph{one-step copy extension} $\Pi'$ of $\Pi$ just as in \cite[***3.19]{fullnorm}. \cite[***Lemma 3.20]{fullnorm} holds here, so $\Pi':(\Tt,\alpha+2)\hookrightarrow_{\min}\Xx$ and $\Pi'$ is standard.\end{rem}

\begin{rem}\label{dfn:inflationary_extender}\label{lem:E-inflation_is_minimal}
The $\Tt$-inflationary case
is likewise a direct generalization: We define \emph{minimal $E$-inflation}
as in  \cite[***3.21]{fullnorm},
under circumstances analogous to there,
though in the 3rd enumerated point of \cite[***3.21]{fullnorm},
we replace ``$\crit(E)<\widetilde{\nu}(Q_{\Pi\beta\xi})$''
with ``$\spc(E)<\lgcd(Q_{\Pi\beta\xi})$'',
and in the 4th point,
replace
$\mathscr{D}^{\Xx'}_{\deg}$ with  $\mathscr{D}^{\Xx'}_{\D\mathrm{-}\deg}$.
\cite[Lemma 3.22]{fullnorm} continues to hold here, via a similar proof.
\end{rem}

We now proceed to the definition of a \emph{minimal inflation} of a normal 
iteration tree $\Tt$, following \cite{iter_for_stacks} and \cite{fullnorm}.

\begin{dfn}\label{dfn:min_inflation}
For $m\in\om\cup\{0^-\}$,
 $M$ an $m$-standard premouse
 and $\Tt,\Xx$ putative $m$-maximal on $M$, with $\Xx$ a true tree, we say that $\Xx$ is a \emph{minimal inflation} of $\Tt$ just in case there is $(t,C,C^-,f,\left<\Pi_\alpha\right>_{\alpha\in C})$ with properties as in \cite[Definition 3.23]{fullnorm}, though with $\mathscr{D}^{\Xx}_{\deg}$ replaced by $\mathscr{D}^{\Xx}_{\D\mathrm{-}\deg}$ in its clause 8.
 \end{dfn}
 \begin{rem}
 \cite[Remark 3.24]{fullnorm} and \cite[Lemmas 3.25, 3.27]{fullnorm} continue to hold,
 and we adopt the notation in \cite[Definition 3.26]{fullnorm}.
\end{rem}

The strategy condensation notions are the direct generalizations of their counterparts in \cite{fullnorm}.

\begin{dfn}\label{dfn:minimal_inflation_condensation}\label{dfn:mhc}
For $\Omega>\om$  regular, $m\in\omega\cup\{0^-\}$,
and $\Sigma$ an $(m,\Omega+1)$-strategy for an $m$-standard pm $M$,
we define \emph{minimal inflation condensation \tu{(}mic\tu{)}} just as in \cite[Definition 3.28]{fullnorm},
and \emph{minimal hull condensation} as in \cite[Definition 3.31]{fullnorm}.\end{dfn}

\begin{rem}\label{lem:inflationary_limit_continuation}\label{lem:unique_strat_has_inf_cond}
The lemmas
\cite[3.29, 3.30, 3.32***]{fullnorm}  hold in the present context,
via the same proofs.
\cite[Theorem 3.34]{fullnorm} also holds, with a slight tweak:
\end{rem}

\begin{tm}\label{tm:wDJ_implies_cond}
Let $\Omega>\om$ be regular and $m\in\omega\cup\{0^-\}$.
Let $M$ be an $m$-standard pure $L[\es]$-premouse
with $\card(M)<\Omega$.
Suppose that if $M$ is active short and $m=0$
then $\rho_{\D}^M=0$.\footnote{We make this assumption because this is the context in which we consider weak Dodd-Jensen when $M$ is active short and $m=0$. If $M$ is active short and $0$-standard but $\rho_{\D}^M>0$,
then we can have iteration maps $i:M\to N$ which arise from $0$-maximal trees which do not drop in any way along $b^\Tt$, but $i^\Tt$ fails to be a $0$-embedding (although it is $0$-deriving).} Let $\Sigma$ be an $(m,\Omega+1)$-strategy for $M$ such that either
$\Sigma$ has the Dodd-Jensen property, or $M$ is countable and $\Sigma$ has the weak Dodd-Jensen property \tu{(}with respect to some enumeration of $M$\tu{)}.
Then $\Sigma$ has minimal hull condensation.
\end{tm}
\begin{proof}
Like for \cite[Theorem 3.34]{fullnorm}, the proof is a routine adaptation of the proof of \cite[Theorem 4.47]{iter_for_stacks}, incorporating the techniques for analysing comparison at the level of $\kappa^+$-supercompactness, such as in \cite{nsp1} and \cite{nsp1fs},
or alternatively in \cite{kappa-plus_v3}.
\end{proof}

\begin{dfn}\label{dfn:inflation_notation}
We also adopt the terminology \emph{$(\Tt)$-pending}, \emph{non-$(\Tt)$-pending},
\emph{$(\Tt)$-terminal},
\emph{$(\Tt)$-terminally-non-dropping}, \emph{$(\Tt)$-terminally-non-dropping}, and define $\pi_\infty^{\Tt\mininflatearrow\Xx}$ just like in \cite[Definition 3.35]{fullnorm}.
\end{dfn}

\begin{rem}\label{rem:non-drop_inf_comm}\cite[Remark 3.36]{fullnorm} continues to hold, though recall that we can have $m=0^-$ or $\deg^\Xx_\infty=0^-$,
and a $0^-$-maximal embedding
between active short premice need not preserve critical points.\end{rem}

\subsection{The factor tree $\Xx/\Tt$}

\begin{rem}\label{lem:<_Xx/Tt}\label{lem:M^Xx^alpha_infty_ult_rep}\label{lem:flattening_properties}
We adopt the notation and terminology introduced in the first paragraph of \cite[\S4.1]{fullnorm}, immediately prior to \cite[Lemma 4.1]{fullnorm}.
(But of course, minimal inflations in the sense of this paper replace
the inflations of \cite{iter_for_stacks}.)
Then \cite[Lemmas 8.4, 8.6]{iter_for_stacks} hold,
where in \cite[Lemma 8.4(iii)]{iter_for_stacks}, we replace
``$\crit(E^{\Xx}_{\zeta^\gamma})<\iota(\exit^{\Xx}_{\zeta^\eta})$''
with ``$\spc(E^{\Xx}_{\zeta^\gamma})<\lgcd(\exit^{\Xx}_{\zeta^\eta})$''.
 \cite[Lemma 4.1]{fullnorm} was a reformulation of
 \cite[8.7]{iter_for_stacks},
using the notation $\xi^\alpha_\kappa$,
 as opposed to the
 notation $\gamma^\alpha_{\theta\kappa}$ of \cite{iter_for_stacks}.
This lemma continues to hold,
except that its clause 4(d) becomes the following,
using the notation $\xi^\alpha_\varsigma=\xi_{\Pi_\alpha\varsigma}$
and $\pi^\alpha_\varsigma=\pi_{\Pi_\alpha\varsigma}$
from Definition \ref{dfn:pi_kappa}:
\begin{enumerate}[label=4(d)]
\item If $\theta+1<\lh(\Tt)$ then for $\varsigma<\lgcd(\exit^\Tt_\theta)$ with $\theta=\alpha^\Tt_\varsigma$,
if $\pi^\alpha_\varsigma(\varsigma)<\spc(E^\Xx_{\zeta^\xi})$ then $\xi^\alpha_\varsigma=\xi^\beta_\varsigma$, and if $\pi^\alpha_\varsigma(\varsigma)\geq\spc(E^\Xx_{\zeta^\xi})$ then $\xi^\alpha_\varsigma=\chi<^\Xx\xi^\beta_\varsigma$.
\end{enumerate}
(Note here that  in \cite{fullnorm}, in the case of interest, $\varsigma$
was is potential critical point $\kappa$ of an extender, whereas
here, $\varsigma$ is a potential space of an extender.)

\cite[Lemma 4.2]{fullnorm} continues to hold, by the same proof.
We define $\pi^{\alpha\beta}$
and $\pi^{*\alpha\beta}$
as in \cite[Definition 4.3]{fullnorm},
replacing $\mathscr{D}^{\Xx}_{\deg}$ with $\mathscr{D}^{\Xx}_{\D\mathrm{-}\deg}$ in its second paragraph. We define the factor tree $\Xx/\Tt$ as in \cite[Definition 4.4]{fullnorm}.

\cite[Lemma 4.5]{fullnorm} carries over, except that
all references to dropping ``in model or degree'' become dropping ``in any way''
(in particular, all instances of ``$\mathscr{D}_{\deg}$''
become ``$\mathscr{D}_{\D\mathrm{-}\deg}$'').
Its proof is the same, except that in the inductive proof of parts 2--5, in the successor step (where $\eta=\alpha+1$), letting $\beta=\pred^\Uu(\alpha+1)$,
$\beta$ is now least such that $\spc(E)<\lgcd(\exit^\Uu_\beta)$, and so the fact that
$\beta=\pred^{\Xx/\Tt}(\alpha+1)$ uses
the modification of \cite[Lemma 8.4(iii)]{iter_for_stacks}  mentioned above.
\end{rem}

\subsection{Minimal comparison}
\begin{rem}\label{dfn:min_comp}
We define minimal comparison just as in
\cite[Definition 5.1]{fullnorm},
and \cite[Lemma 5.2]{fullnorm} holds here,
by the same proof.
\end{rem}

\subsection{Minimal inflation stacks}
\begin{rem}\label{lem:inflation_commutativity}\label{cor:terminal_dropping_equiv}
The commutativity of minimal inflation (\cite[Lemma 6.1]{fullnorm}) goes through,
replacing $\mathscr{D}^{\Xx_2}_{\deg}$ with $\mathscr{D}^{\Xx_2}_{\D\mathrm{-}\deg}$ in clause C5(d),(e). The proof is like there, except that in the proof of part C5(f) in Case 3, the inequality ``$\crit(E)<\crit(E^{\Xx_2}_{\alpha_2})<\crit(F)$''
should be replaced with ``$\spc(E)<\spc(E^{\Xx_2}_{\alpha_2})<\spc(F)$''.
(In fact $\spc(E)<\lambda(E)\leq\crit(E^{\Xx_2}_{\alpha_2})\leq\spc(E^{\Xx_2}_{\alpha_2})<\spc(F)$,
but it is possible that $\spc(E^{\Xx_2}_{\alpha_2})=\crit(E^{\Xx_2}_{\alpha_2})=\crit(F)<\spc(F)$.) And \cite[Corollary 6.2]{fullnorm} holds.

\cite[\S6.2]{fullnorm} adapts directly: we define \emph{\tu{(}continuous\tu{)} terminal minimal inflation stack} as in \cite[Definition 6.3]{fullnorm} and \emph{goodness}
as in \cite[Definition 6.5]{fullnorm}. \cite[Lemmas 6.4, 6.6]{fullnorm} go through with identical proofs, just replacing references to dropping in model or degree to dropping in any way.
\end{rem}

\subsection{Normalization of transfinite stacks}\

Theorems \ref{tm:full_norm_stacks_strategy} and
\ref{tm:stacks_iterability_2}
are now proved just like \cite[***Theorems 1.1, 7.2]{fullnorm} (in \cite[***\S7]{fullnorm}),
replacing drops ``in model or degree'' with drops ``of any kind''. Likewise for the ``minimal'' adaptation of \cite[Theorem 9.6]{iter_for_stacks} to our context.

\subsection{Anlaysis of comparison}\label{sec:comparison_analysis}

The analaysis of comparison in
\cite[\S8]{fullnorm} goes through in the same fashion.

\bibliographystyle{plain}
\bibliography{../bibliography/bibliography}

\end{document}